\documentclass[sn-mathphys,13pt]{sn-jnl}% Math and Physical Sciences Reference Style
\usepackage{amssymb}
\usepackage{amsmath}
\usepackage{amsfonts}
\usepackage{mathrsfs}
\usepackage{geometry}
\usepackage{enumerate}
\usepackage{enumitem}
\setitemize[1]{itemsep = 0pt, partopsep = 0pt, parsep = 0pt, topsep = 5pt, labelsep =  0pt}
\allowdisplaybreaks
\jyear{2021}%
\theoremstyle{thmstyleone}%
\newtheorem{theorem}{Theorem}[section]%%  meant for continuous numbers
\newtheorem{lemma}{Lemma}[section]%
\theoremstyle{thmstyletwo}%
\newtheorem{remark}{Remark}%
\theoremstyle{thmstylethree}%
\numberwithin{remark}{section}
\begin{document}

\title[Existence of the solution for a double phase system]{Multiple positive solutions for a double phase system with singular nonlinearity}


\author[1]{\fnm{Yizhe} \sur{Feng}}\email{yzfeng2021@163.com}

\author*[1]{\fnm{Zhanbing} \sur{Bai}}\email{zhanbingbai@163.com}
\equalcont{These authors contributed equally to this work.}
\affil[1]{\orgdiv{College of Mathematics and System Science}, \orgname{Shandong University of Science and Technology}, \city{Qingdao}, \postcode{266590}, \state{Shandong}, \country{China}}


\abstract{ In this paper, we study a class of double phase systems which contain the singular and mixed nonlinear terms. Unlike the single equation, the mixed nonlinear terms make the problem more complicate. The geometry of the fibering mapping has multiple possibilities. To overcome the difficulties posed by the mixed nonlinear terms, we need to repeatedly construct concave functions, discuss different cases, and use the properties of concave functions and basic inequalities such as H\"{o}lder inequality, Poincar\'{e}'s inequality and Young's inequality. By the use of the Nehari manifold, the existence and multiplicity of positive solutions which have nonnegative energy are obtained. It is worth mentioning that we note the existence of saddle point solution(a station point that is not a local minimum), see Remark \ref{reandianjie}.}




\maketitle

\section{Introduction}
In this work, the multiplicity of solutions to a class of double phase systems with $m$ equations and Dirichlet boundary value condition of the type
\begin{equation}\label{eq}
\begin{cases}
    - \Delta_p y_1 - {\rm div}(\eta\lvert \nabla y_1\rvert^{q - 2}\nabla y_1)\\
   \ \ \ \ \ \ \ \  = a_1(z)\vert y_1\vert ^{ -1- \nu}y_1 + \lambda(\kappa_1 + 1)\lvert y_1\lvert^{\kappa_1 - 1}y_1 \lvert y_2 \lvert^{\kappa_2 + 1}\cdots \lvert y_m \lvert^{\kappa_m + 1},  & z\in \Omega, \\
    - \Delta_p y_2 - {\rm div}(\eta\lvert \nabla y_2\rvert^{q - 2}\nabla y_2)\\
  \ \ \ \ \ \ \ \   = a_2(z)\vert y_2\vert ^{ -1- \nu}y_2 + \lambda(\kappa_2 + 1)\lvert y_1 \lvert ^{\kappa_1 + 1}\lvert y_2\lvert^{\kappa_2 - 1}y_2\cdots \lvert y_m \lvert^{\kappa_m + 1} , & z\in \Omega, \\
  \ \ \ \ \ \ \ \  \vdots\\
    - \Delta_p y_m - {\rm div}(\eta\lvert \nabla y_m\rvert^{q - 2}\nabla y_m)\\
   \ \ \ \ \ \ \ \  = a_m(z)\vert y_m\vert ^{ -1- \nu}y_m+ \lambda(\kappa_m + 1)\lvert y_1 \lvert ^{\kappa_1 + 1}\lvert y_2\lvert^{\kappa_2 + 1}\cdots \lvert y_m \lvert^{\kappa_m - 1}y_m , & z\in \Omega, \\
   y_1 = y_2 = \cdots = y_m = 0, &z\in \partial \Omega
\end{cases}
\end{equation}
is investigated, where $m\geq 2$, $\lambda>0$ and $0 < \nu < 1 < p < q < \kappa_i + 1 < p^*$ ($i$ = 1, 2, $\cdots$, $m$); $\Delta_py = {\rm div}(\lvert \nabla y\rvert^{p - 2}\nabla y)$; $\Omega \subseteq \mathbb{R}^n $, $n\geq 2$ is a bounded domain with Lipschitz boundary $\partial \Omega$.

In 1986, Zhikov \cite{Zhikov1} first introduced the double phase operator that appear in \eqref{eq} which is denoted by
\begin{equation}\label{doublephase}
  y \mapsto  {\rm div}(\lvert \nabla y\rvert^{p - 2}\nabla y + \eta(z)\lvert \nabla y\rvert^{q - 2}\nabla y), \ y\in W^{1, \mathcal{H}},
\end{equation}
in order to describe the hardening properties of strongly anisotropic materials which change drastically at some of their
points. Zhikov introduced the related energy functional defined by
\begin{equation}\label{dp - energy}
y \mapsto \int_{\Omega}\left( \frac{1}{p}\lvert \nabla y\rvert^p + \frac{\eta(z)}{q}\lvert \nabla y\rvert^{q} \right)dz.
\end{equation}
The variational integral functional \eqref{dp - energy} intervene in Homogenization theory and Elasticity, and also be used to describe the new examples of Lavrentiev’s phenomenon \cite{Zhikov2, Zhikov3}. The energy density
\begin{equation*}
 \rho_{(p, q)}(y, z) =  \frac{1}{p}\lvert y\rvert^p + \frac{\eta(z)}{q}\lvert y\rvert^{q}
\end{equation*}
of \eqref{dp - energy} change their ellipticity rate according to the point, the geometry of a composite made of two materials with their power hardening exponents $p$ and $q$, respectively, are depended on the modulation coefficient $\eta(\cdot)$. The ``$(p, q)\mbox{-}$phase'' refers to $\eta(z) > 0$, $q>p$, and $\rho$ exhibits a polynomial growth of order $q$ with respect to the gradient variable $z$. The growth is at rate $p$ when ${\eta(z) = 0}$ and this is called ``$p$$\mbox{-}$phase''. Subsequently, many scholars have studied the functionals in \eqref{doublephase}, Marcellini \cite{Marce1, Marce2} proved \eqref{doublephase} also belongs to the integrative functional class with non$\mbox{-}$standard growth conditions, and other. We provide readers the works of Baroni, et al \cite{Baroni1, Baroni2}, Colombo and Mingione \cite{Colombo1, Colombo2} to learn more.

The problem of form \eqref{doublephase} also appeared in many physics issues, notably the flow of non$\mbox{-}$Newtonian fluids. Liu and Dai \cite{LD2} used the motion of a non$\mbox{-}$Newtonian fluid between two planks as a model to explore the relationship between the double phase operator and internal friction. Denote by $v$ the speed of this fluid at some layer, $F$ the internal frictional force in the opposite direction to $v$, $S$ the contact area between the plate and the fluid, and the viscosity coefficient $\mu$ is defined by $\mu := \lvert \nabla y\rvert^{p - 2} + \eta(z)\lvert \nabla y\rvert^{q - 2}$. By using Newton's viscosity law, he obtained
\begin{equation}\label{FS}
  \frac{{\rm div} \vec{F}}{S} = {\rm div}((\lvert \nabla y\rvert^{p - 2} + \eta(z)\lvert \nabla y\rvert^{q - 2})\nabla y).
\end{equation}
Let $F\equiv0$ in \eqref{FS}, and is actually the Euler equation of functional \eqref{doublephase}. In addition, double phase operator also appears in the study of
torsional creep \cite{Kaw} and glacial sliding \cite{SW}.

In 2018, Liu and Dai \cite{LD1} used variational method to obtain the existence and multiplicity of solutions for the following double phase problem in  Musielak$\mbox{-}$Orlicz spaces
\begin{equation}\notag
  \begin{cases}
    - \Delta_py - {\rm div}(\eta(z)\lvert \nabla y\rvert^{q - 2}\nabla y) = f(z, y), & z\in \Omega, \\
    y = 0, & z\in \partial \Omega.
\end{cases}
\end{equation}
In 2022, Liu and Dai \cite{LD4} studied the above problem where
\begin{equation*}
  f(z, y) = a(z)y^{ - \nu} + \lambda y^{r - 1}, \ \ 0 < \nu < 1, \ p< q < r < p^*.
\end{equation*}
It is obvious that $f$ is no longer a Carathéodory function (because of the singular term). By using Nehari manifold, the existence of two positive solutions of the problem has been proved. Then in \cite{LD3}, they studied the same equation as \cite{Papa} with $z\in\mathbb{R}^{n}$ instead of in the bounded domain $\Omega$. We present readers \cite{Farkas1,Farkas2,Arora, Zeng1, Zeng2, Gas1,Gas2, Bahrouni, Cre, CPW, ColS, Byun} to learn more about the existence and multiplicity results of double phase problems.

However, the above research on the double phase problem is with respect to a single equation, only a few system of equations for the double phase problem has been studied. In 2021, Bahrouni and Rădulescu \cite{A. Bahrouni} studied the following singular double phase system with variable growth and baouendi-grushin operator
\begin{equation*}
   \begin{cases}
    - \Delta_{G(z_1,z_2)}y_1 + \vert y_1 \vert^{q(z)-2}y_1 + \vert y_1 \vert^{p(z)-2}y_1 = a_1y_1^{-\nu_1}-b\alpha\vert y_2\vert^{\beta}\vert y_1\vert^{\alpha-2}y_1,   \\
     - \Delta_{G(z_1,z_2)}y_2 + \vert y_2 \vert^{q(z)-2}y_2 + \vert y_2 \vert^{p(z)-2}y_2 = a_2y_2^{-\nu_2}-b\beta\vert y_1\vert^{\alpha}\vert y_2\vert^{\beta-2}y_2, \\
\end{cases}
\end{equation*}
where $z=(z_1,z_2)\in \mathbb{R}^{n}$, $a_1,a_2,b,p,q,\alpha,\beta \in C(\mathbb{R}^{n},\mathbb{R})$, $\nu_1,\nu_2:\mathbb{R}^{n}\rightarrow (0,1)$. He established a related compactness property, and obtained the existence of at least one weak solution. For double phase systems containing convection terms, Marino and Winkert in \cite{Marino G} studied the following double phase system
\begin{equation}\label{convention}
  \begin{cases}
    - \Delta_{p_1}y_1 - {\rm {\rm div}}(\eta_1(z)\lvert \nabla y_1\rvert^{q_1 - 2}\nabla y_1) =f_1(z,y_1,y_2,\nabla y_1,
    \nabla y_2), & z\in \Omega, \\
    - \Delta_{p_2}y_2 - {\rm div}(\eta_2(z)\lvert \nabla y_2\rvert^{q_2 - 2}\nabla y_2) = f_2(z,y_1,y_2,\nabla y_1,
    \nabla y_2), & z\in \Omega, \\
     y_1 = y_2 = 0, & z\in \partial \Omega.
  \end{cases}
\end{equation}
The existence and uniqueness of the weak solution of the system \eqref{convention} is obtained by using the surjectivity result for pseudomonotone operators. Then, in 2022, Guarnotta, et al \cite{Guarnotta U} studied \eqref{convention} with variable growth and nonlinear boundary condition. By using the sub-supersolution method, they obtained infinitely many solutions. The methods of the above works are completely different from ours.

Motivated by above research, in this paper, we consider the study of a system shown in \eqref{eq} with singular and nonlinear terms. Unlike the single equation, the mixed nonlinear terms can complicate the problem, some new ideas and techniques are needed.
\begin{itemize}[itemindent = -0.5em]
  \item[$(1)$]\ The geometry of the fibering mapping has multiple possibilities, we need to discuss the mixed nonlinear terms in different cases depending on their energies to determine the geometry of the fibering mapping.
  \item[$(2)$]\ To prove that the limit of the minimizing sequence is not zero, we discuss different cases and prove it by the combination of a series of inequalities, instead of getting the conclusion directly from one inequality.
\end{itemize}
 Also, To overcome the difficulties posed by the mixed nonlinear terms, we repeatedly construct concave functions, discuss different cases, and use the properties of concave functions and basic inequalities such as H\"{o}lder inequality, Poincar\'{e}'s inequality and Young's inequality.

%We will use the Nehari manifold method to obtain the existence results for two positive solutions. The Nehari manifold used in this paper was first introduced by Nehari \cite{Nehari} to solve the boundary value problem of a class of second$\mbox{-}$order nonlinear elliptic equations, we provide reader \cite{Nehari,Bozhkov,Drabek, Brown} to learn more about the Nehari manifolds and fibering map. Unlike elliptic systems of the general form, we are
%working in Musielak$\mbox{-}$Orlicz Sobolev vector spaces instead of usual Sobolev vector spaces.

For the sake of simplicity, we can concentrate the argument on the case that $m=2$. In Section 2, we recall the main properties on the theory of Musielak$\mbox{-}$Orlicz spaces and give some definitions of the fibering map and the Nehari manifold. In Section 3, some lemmas and propositions are given which are required for the existence results, and then we prove the main existence theorem (see Theorem \ref{maintheorem}). Finally, in Theorem \ref{fangchengzujie} we generalize the results of Theorem \ref{maintheorem} to the case that $m>2$.

\par The following hypotheses $(H_1)-(H_3)$ will be assumed,
\begin{itemize}[itemindent = 0.4em]
  \item[$(H_1)$: ]$1 < p < n,$ $p < q < p^*$, where $p^*= \frac{np}{n - p}$ and $\eta : \overline{\Omega} \rightarrow [0,\infty)\in L^{\infty}(\Omega)$ and $\eta(z) \geq 0$;
  \item[$(H_2)$: ]$a_i\in L^\infty(\Omega), a_i(z)>0$ for $a.e. ~z\in \Omega, i = 1, ..., m$;
  \item[$(H_3)$: ]$\sum_{i = 1}^{m} \kappa_i + m < p^*$.
\end{itemize}

\section{Preliminaries}
In this section we recall some results of Musielak$\mbox{-}$Orlicz spaces. These results are from \cite{LD1,LD4,Cre,ColS}.

Let $W_{0}^{1, p}(\Omega)$ be the Sobolev space equipped with the norm
\begin{equation*}
  \|y\|_{1, p} = \left(\int_{\Omega}\vert \nabla y\vert^pdz\right)^{\frac{1}{p}}.
\end{equation*}
For $(y_1, y_2)\in W_{0}^{1, p}(\Omega)\times W_{0}^{1, p}(\Omega)$, let
\begin{equation*}
  \| (y_1, y_2)\|_{1,p} = \|y_1\|_{1, p}^p + \|y_2\|_{1, p}^p.
\end{equation*}
Denote by $\|\cdot\|_p$ the norm of $L^p(\Omega)$. Then one has
\begin{equation*}
\| (y_1, y_2)\|_{1,p} = \|y_1\|_{1, p}^p + \|y_2\|_{1, p}^p = \| \nabla y_1\|_p^p + \|\nabla y_2\|_p^p.
\end{equation*}
Denote the maximum norm of $\mathbb{R}^n$ by $\vert \cdot \vert$.

Let $L^{\mathcal{H}}(\Omega)$ defined as
\begin{equation*}
  L^{\mathcal{H}}(\Omega)=\left\{y\big\vert y: \Omega\rightarrow \mathbb{R} {\rm\ is\ measurable\ and\ }\rho_{\mathcal{H}}(y)<+\infty\right\}
\end{equation*}
 be the Musielak$\mbox{-}$Orlicz space with Luxemburg norm
\begin{equation*}
  \|y\|_{\mathcal{H}}=\inf \left\{ \tau > 0: \rho_{\mathcal{H}}\left(\frac{y}{\tau}\right)\leq 1 \right\},
\end{equation*}
where $\mathcal{H}:\Omega\times [0,\infty) \rightarrow [0,\infty)$ be the function defined as
\begin{equation*}
  \mathcal{H}(z,t) = t^p + \eta(z)t^q,
\end{equation*}
 and the modular function $\rho_{\mathcal{H}}:L^{\mathcal{H}}(\Omega)\rightarrow \mathbb{R}$ is given by
\begin{equation}\label{2.1(liu)}
  \rho_{\mathcal{H}}(y):=\int_{\Omega} \mathcal{H}(z,\vert y(z) \vert)dz=\int_{\Omega}(\vert y(z) \vert^p + \eta(z)\vert y(z) \vert^q)dz.
\end{equation}
The space $L^{\mathcal{H}}(\Omega)$ is a reflexive Banach, see Proposition 2.14 of [13].

 The seminormed space $ L_\eta^q(\Omega)$ is defined as
\begin{equation*}
  L_\eta^q(\Omega)=\left\{y \bigg\vert y:\Omega\rightarrow \mathbb{R} {\rm\ is\ measurable\ and\ } \int_{\Omega}\eta(z)\vert y(z) \vert^{q} dz< +\infty\right\},
  \end{equation*}
endowed with the seminorm
\begin{equation*}
  \| y\|_{q,\eta}= \left(\int_{\Omega} \eta(z) \lvert y(z) \rvert^{q}dz \right)^{\frac{1}{q}}.
\end{equation*}
The space $L_{\eta}^q(\Omega;\mathbb{R}^N)$ is defined as well.

Let $W^{1,\mathcal{H}}(\Omega)$ be the Musielak$\mbox{-}$Orlicz Sobolev space which defined as
\begin{equation*}
 W^{1,\mathcal{H}}(\Omega)=\left\{ y\in L^{\mathcal{H}}(\Omega) : \vert \nabla y \vert \in L^{\mathcal{H}}(\Omega)\right\}
\end{equation*}
equipped with the norm
\begin{equation*}
  \|y\|_{1,\mathcal{H}}=\|\nabla y\|_{\mathcal{H}} + \| y\|_{\mathcal{H}},
\end{equation*}
where $\|\nabla y\|_{\mathcal{H}} = \|\vert\nabla y\vert\|_{\mathcal{H}}$. Let $W_{0}^{1,\mathcal{H}}(\Omega)$ be defined as the completion of $C_{0}^{\infty}(\Omega)$ in $W^{1,\mathcal{H}}(\Omega)$. Thanks to Proposition 2.16(ii) of Crespo–Blanco \cite{Cre} or Proposition 2.2 of R. Arora \cite{Arora2}, we know $\|\nabla y\|_{\mathcal{H}}$ is an equivalent norm on $W_{0}^{1,\mathcal{H}}(\Omega)$ when assumption $(H_1)$ be satisfied.

Furthermore, both $W^{1,\mathcal{H}}(\Omega)$ and $W_{0}^{1,\mathcal{H}}(\Omega)$ are uniformly concave Banach spaces, see Proposition 2.14 and Proposition 2.18(ii) of Crespo–Blanco \cite{Cre}.

For $(y_1, y_2)\in W_{0}^{1, \mathcal{H}}(\Omega)\times W_{0}^{1, \mathcal{H}}(\Omega)$, let
\begin{equation*}
  \|(y_1,y_2)\|= \|\nabla y_1\|_{\mathcal{H}} + \|\nabla y_2\|_{\mathcal{H}} ,
\end{equation*}
and
\begin{equation*}
  \| (y_1, y_2)\|_{q,\eta}=\| y_1\|_{q,\eta}^{q}+\| y_2\|_{q,\eta}^{q}.
\end{equation*}
Thus, it is clearly that
\begin{equation*}
  \rho_{\mathcal{H}}(\nabla y_1) + \rho_{\mathcal{H}}(\nabla y_2) = \| (y_1, y_2)\|_{1,p} + \| (\nabla y_1, \nabla y_2)\|_{q,\eta}.
\end{equation*}

The following embedding results from Propositions 2.17 and 2.19 of Crespo–Blanco \cite{Cre} or Proposition 2.2 of R. Arora \cite{Arora2} are critical to our article.

\begin{lemma}(Proposition 2.2, \cite{Arora2})\label{prop2.2(Aro)}
Let $(H_1)$ be satisfied. Then the following embeddings hold:
\begin{itemize}[itemindent = 0.2em]
  \item[$(i)$:]\ \ $L^{\mathcal{H}}(\Omega) \hookrightarrow L^{r}(\Omega)$ and $W_{0}^{1,\mathcal{H}}(\Omega) \hookrightarrow W_{0}^{1,r}(\Omega)$ are continuous for all $r \in [1,p]$;
  \item[$(ii)$:]\ \ $W_{0}^{1,\mathcal{H}}(\Omega) \hookrightarrow L^{r}(\Omega)$ is continuous for all $r \in [1,p^*]$;
  \item[$(iii)$:]\ \ $W_{0}^{1,\mathcal{H}}(\Omega) \hookrightarrow L^{r}(\Omega)$ is compact for all $r \in [1,p^*)$;
  \item[$(iv)$:]\ \ $L^{\mathcal{H}}(\Omega) \hookrightarrow L_{\eta}^{q}(\Omega)$ is continuous;
  \item[$(v)$:]\ \ $L^{q}(\Omega) \hookrightarrow L^{\mathcal{H}}(\Omega)$ is continuous.
 \end{itemize}
\end{lemma}

\begin{lemma}(Proposition 2.1, \cite{LD1})\label{prop2.2(liu)}
The $\mathcal{H}$$\mbox{-}$modular has the following properties
\begin{itemize}[itemindent  =  0.2em]
  \item[$(i)$:]\ \ For $y\neq 0$, then $\|y\|_{\mathcal{H}} = \lambda\Leftrightarrow \rho_{\mathcal{H}}(\frac{y}{\lambda}) = 1$;
  \item[$(ii)$:]\ \ $\|y\|_{\mathcal{H}}<1$ (resp. $>1$; $ = 1)\Leftrightarrow\rho_{\mathcal{H}}(y)<1$ (resp. $>1$; $ = 1$);
  \item[$(iii)$:]\ \ $\|y\|_{\mathcal{H}}<1\Rightarrow\|y\|_{\mathcal{H}}^{q} \leq \rho_{\mathcal{H}}(y) \leq \|y\|_{\mathcal{H}}^{p}$; $\|y\|_{\mathcal{H}}>1\Rightarrow\|y\|_{\mathcal{H}}^{p} \leq \rho_{\mathcal{H}}(y) \leq \|y\|_{\mathcal{H}}^{q}$;
  \item[$(iv)$:]\ \ $\|y\|_{\mathcal{H}}\rightarrow 0\Leftrightarrow\rho_{\mathcal{H}}(y) \rightarrow 0$; $\|y\|_{\mathcal{H}}\rightarrow  + \infty\Leftrightarrow\rho_{\mathcal{H}}(y) \rightarrow  + \infty$.
 \end{itemize}
\end{lemma}

%
%Let $A:W_{0}^{1,\mathcal{H}}(\Omega)\rightarrow W_{0}^{1,\mathcal{H}}(\Omega)^*$ be the nonlinear map defined by
%\begin{equation}\label{2.2(liu)}
%  \left<A(u),\phi \right>_{\mathcal{H}}:= \int_{\Omega}\vert \nabla u\vert^{p - 2}(\nabla u, \nabla \phi)dz + \int_{\Omega}\eta\vert \nabla u\vert^{q - 2}(\nabla u, \nabla \phi)dz
%\end{equation}
%for all $u,\phi\in W_{0}^{1,\mathcal{H}}(\Omega)$, where $\left<\cdot, \cdot\right>$ is the duality pairing between $W_{0}^{1,\mathcal{H}}(\Omega)$ and its dual space $W_{0}^{1,\mathcal{H}}(\Omega)^*$. The operator $A$ has the following properties:
%\begin{proposition}(Proposition 2.3, \cite{LD4})
%  The operator $A$ defined by \eqref{2.2(liu)} is bounded (that is, it maps bounded
%sets into bounded sets), continuous, strictly monotone (hence maximal monotone) and
%it is of type $(S^+)$.
%\end{proposition}

\begin{lemma}(Theorem A.3.2, \cite{Wangmx}, Poincar\'{e}'s inequality)
  Let $\Omega\subset\mathbb{R}^{n}$ be a bounded open set and $1\leq p < n$, then for given $q\in[1, p^*]$, there exists $C = C(n, p, q, \Omega) > 0$ such that
  \begin{equation*}
    \|y\|_{q}\leq C\|\nabla y\|_p, \ \ \forall\ y\in W_{0}^{1, p}(\Omega).
  \end{equation*}
\end{lemma}
\begin{lemma}(Theorem A.5.1, \cite{Wangmx})\label{jinqianru}
  Let  $\Omega\subset\mathbb{R}^{n}$ be a bounded open set with $C^{1}$ boundary $\partial\Omega$ and $1\leq p < n$, then for any $1\leq q < p^*$,
  \begin{equation*}
    W_{0}^{1, p}(\Omega)\hookrightarrow\hookrightarrow L^{q}(\Omega).
  \end{equation*}
\end{lemma}

We say that $(y_1, y_2)\in W_{0}^{1, \mathcal{H}}(\Omega)\times W_{0}^{1,\mathcal{H}}(\Omega)$ is a weak solution of \eqref{eq}, if for any $(h, w)\in W_{0}^{1, \mathcal{H}}(\Omega)\times W_{0}^{1, \mathcal{H}}(\Omega)$, there holds $\left(a_1y_1^{ - \nu}h, a_2y_2^{ - \nu}w\right)\in L^{1}(\Omega)\times L^{1}(\Omega)$ for $y_1,y_2>0$  and
\begin{equation*}
  \begin{split}
  \int_{\Omega}&\vert \nabla y_1\vert^{p - 2}(\nabla y_1, \nabla h)dz + \int_{\Omega}\eta\vert \nabla y_1\vert^{q - 2}(\nabla y_1, \nabla h)dz - \int_{\Omega}a_1 y_1^{ - \nu}hdz \\
& + \int_{\Omega}\vert \nabla y_2\vert^{p - 2}(\nabla y_2, \nabla w)dz + \int_{\Omega}\eta\vert \nabla y_2\vert^{q - 2}(\nabla y_2, \nabla w)dz - \int_{\Omega}a_2 y_2^{ - \nu}w dz\\
& - \lambda(\kappa_1 + 1)\int_{\Omega}\vert y_1\vert^{\kappa_1}\vert y_2\vert^{\kappa_2 + 1} h dz  - \lambda(\kappa_2 + 1)\int_{\Omega}\vert y_1\vert^{\kappa_1 + 1}\vert y_2\vert^{\kappa_2} w dz = 0.
  \end{split}
\end{equation*}

Let the energy functional $J: W_{0}^{1, \mathcal{H}}(\Omega)\times W_{0}^{1, \mathcal{H}}(\Omega)\rightarrow \mathbb{R}$ be defined as
 \begin{align*}
   J(y_1, y_2) =  &\frac{1}{p}\| (y_1, y_2)\|_{1,p}  + \frac{1}{q}\| (\nabla y_1, \nabla y_2)\|_{q,\eta}-  \frac{1}{1 - \nu}\int_{\Omega}\left[a_1\vert y_1\vert^{1 - \nu} + a_2\vert y_2\vert^{1 - \nu}\right]dz\\
   &-  \lambda\int_{\Omega}\vert y_1\vert^{\kappa_1 + 1} \vert y_2\vert^{\kappa_2 + 1}dz.
 \end{align*}
Then the derivative of $J$ at $(y_1, y_2)$ with direction $(h, w)$ is given by
\begin{equation}\label{J^'uvhw}
\begin{split}
\big < J^{\prime}(y_1, y_2), (h, w)\big >  = & \int_{\Omega}\vert \nabla y_1\vert^{p - 2}(\nabla y_1, \nabla h)dz + \int_{\Omega}\eta\vert \nabla y_1\vert^{q - 2}(\nabla y_1, \nabla h)dz \\
& - \int_{\Omega}a_1\vert y_1\vert^{ -1- \nu}y_1hdz - \lambda(\kappa_1 + 1)\int_{\Omega}\vert y_1\vert^{\kappa_1-1}y_1\vert y_2\vert^{\kappa_2 + 1} h dz\\
& + \int_{\Omega}\vert \nabla y_2\vert^{p - 2}(\nabla y_2, \nabla w)dz + \int_{\Omega}\eta\vert \nabla y_2\vert^{q - 2}(\nabla y_2, \nabla w)dz \\
& - \int_{\Omega}a_2\vert y_2\vert^{-1 - \nu}y_2w dz - \lambda(\kappa_2 + 1)\int_{\Omega}\vert y_1\vert^{\kappa_1 + 1}\vert y_2\vert^{\kappa_2-1}y_2 w dz.\\
\end{split}
\end{equation}

Let Nehari manifold defined by
\begin{equation*}
  \mathcal{N}_{\lambda}  = \left\{ (y_1, y_2)\in W_{0}^{1, \mathcal{H}}(\Omega) \times W_{0}^{1, \mathcal{H}}(\Omega)\backslash\{(0, 0)\}: \big < J^{\prime}(y_1, y_2), (y_1, y_2)\big >  = 0 \right\}.
\end{equation*}
 Obviously, all critical points of $J$ are on the Nehari manifold, so, $\mathcal{N}_{\lambda}$ contains all weak solutions of \eqref{eq}.
In order to better understand Nehari manifold, we define a function $\psi_{(y_1, y_2)}(t) = J(ty_1, ty_2)$ for $(y_1,y_2)\in  W_{0}^{1, \mathcal{H}}(\Omega) \times W_{0}^{1, \mathcal{H}}(\Omega)\backslash\{(0, 0)\}$. Then
%, $\psi_1(t) = J_1(u, tv)$, $\psi_2(t) = J_2(u, tv)$, $t > 0$. Thus
\begin{equation*}\label{partialpsit}
  \begin{split}
  \psi_{(y_1, y_2)}^{\prime}(t) = &t^{p - 1}\| (y_1, y_2)\|_{1,p} + t^{q - 1}\| (\nabla y_1, \nabla y_2)\|_{q,\eta}\\
   & - t^{ - \nu}\int_{\Omega}\left[a_1\vert y_1 \vert^{1 - \nu} + a_2\vert y_2 \vert^{1 - \nu}\right]dz \\
   & -  \lambda(\kappa_1 + \kappa_2 + 2)t^{\kappa_1 + \kappa_2 + 1}\int_{\Omega}\vert y_1\vert^{\kappa_1 + 1}\vert y_2\vert^{\kappa_2 + 1}dz.
\end{split}
\end{equation*}
Hence, we give an equivalent definition of $\mathcal{N}_{\lambda}$ as
\begin{equation}
  \mathcal{N}_{\lambda}  = \left\{ (y_1, y_2)\in W_{0}^{1, \mathcal{H}}(\Omega) \times W_{0}^{1, \mathcal{H}}(\Omega)\backslash\{(0,0)\}: \psi_{(y_1, y_2)}^{\prime}(t) \bigg\vert_{t = 1} = 0 \right\}.
\end{equation}

Furthermore, one has
\begin{equation}\label{2partialpsit}
  \begin{split}
   \psi_{(y_1, y_2)}^{\prime\prime}(t)  = &(p - 1) t^{p - 2}\| (y_1, y_2)\|_{1,p} + (q - 1)t^{q - 2}\| (\nabla y_1, \nabla y_2)\|_{q,\eta}\\
   & + \nu t^{ - \nu - 1}\int_{\Omega}\left[a_1\vert y_1 \vert^{1 - \nu} + a_2\vert y_2 \vert^{1 - \nu}\right]dz\\
   & - \lambda(\kappa_1 + \kappa_2 + 2)(\kappa_1 + \kappa_2 + 1) t^{\kappa_1 + \kappa_2}\int_{\Omega}\vert y_1\vert^{\kappa_1 + 1}\vert y_2\vert^{\kappa_2 + 1}dz.
\end{split}
\end{equation}
Hence, we can divide $\mathcal{N}_{\lambda}$ into three disjoint subsets:
\begin{align*}
  \mathcal{N}_{\lambda}^{ + } = \left\{(y_1, y_2)\in \mathcal{N}_{\lambda}:\psi_{(y_1, y_2)}^{\prime\prime}(t) \bigg\vert_{t = 1} > 0 \right\}, \\
  \mathcal{N}_{\lambda}^{0} = \left\{(y_1, y_2)\in \mathcal{N}_{\lambda}:\psi_{(y_1, y_2)}^{\prime\prime}(t) \bigg\vert_{t = 1} = 0 \right\}, \\
 \mathcal{N}_{\lambda}^{ - } = \left\{(y_1, y_2)\in \mathcal{N}_{\lambda}:\psi_{(y_1, y_2)}^{\prime\prime}(t) \bigg\vert_{t = 1} < 0 \right\}.
\end{align*}
By the definitions of $\mathcal{N}_{\lambda}$, $\mathcal{N}_{\lambda}^{ + }$, $\mathcal{N}_{\lambda}^{ - }$, and the function $\psi_{(y_1, y_2)}(t)$, the following lemma clearly holds.
\begin{lemma}\label{lemma2.3}
 Let $(y_1, y_2)\in W_{0}^{1, \mathcal{H}}(\Omega) \times W_{0}^{1, \mathcal{H}}(\Omega)\backslash\{(0,0)\}$, then for $t > 0, t(y_1, y_2)\in\mathcal{N}_{\lambda}$ if and only if $\psi_{(y_1, y_2)}^{\prime}(t) = 0; ~ t(y_1, y_2)\in\mathcal{N}_{\lambda}^{ + }$ if and only if $\psi_{(y_1, y_2)}^{\prime}(t) = 0$ and $\psi_{(y_1, y_2)}^{\prime\prime}(t) > 0; ~ t(y_1, y_2)\in\mathcal{N}_{\lambda}^{ - }$ if and only if $\psi_{(y_1, y_2)}^{\prime}(t) = 0$ and $\psi_{(y_1, y_2)}^{\prime\prime}(t) < 0$.
\end{lemma}

We will prove that when the parameter $\lambda$ within a certain range, the two solutions of \eqref{eq} are in sets $\mathcal{N}_{\lambda}^{ + }$ and $\mathcal{N}_{\lambda}^{ - }$, respectively.

\section{The existence and multiplicity}

Firstly, we study the properties of the energy functional $J$ on $\mathcal{N}_{\lambda}$, $\mathcal{N}_{\lambda}^{ + }$, $\mathcal{N}_{\lambda}^{ - }$, respectively, and  prove that $\mathcal{N}_{\lambda}^{0} = \emptyset$ when $\lambda$ is small enough. Secondly, we prove the existence of convergent subsequences on $\mathcal{N}_{\lambda}^{ + }$ and $\mathcal{N}_{\lambda}^{-}$, respectively. Thirdly, it is proved that both two convergence points in $\mathcal{N}_{\lambda}^{ + }$ and $\mathcal{N}_{\lambda}^{-}$, respectively, are the solutions of \eqref{eq}. Finally, we generalize the results of the system \eqref{eq} to the case of $m>2$.

\begin{lemma}\label{bddblow}
Suppose $\lambda  > 0$ and assumptions $(H_1)$, $(H_2)$ hold, then $J(y_1, y_2)\vert_{\mathcal{N}_{\lambda}}$ is coercive.
\end{lemma}
\begin{proof}
From the definition of the Nehari manifold, one has
\begin{equation*}
  \begin{split}
     \lambda(\kappa_1 + \kappa_2 + 2)\int_{\Omega}\vert y_1\vert^{\kappa_1 + 1}\vert y_2\vert^{\kappa_2 + 1}dz =& \| (y_1, y_2)\|_{1,p} +\|(\nabla y_1,\nabla y_2)\|_{q,\eta}\\
     & - \int_{\Omega}\left[a_1\vert y_1 \vert^{1 - \nu} + a_2\vert y_2 \vert^{1 - \nu}\right]dz.
  \end{split}
\end{equation*}

Taking into account that $(H_1)$, $(H_2)$ hold, and the fact $p < q < \kappa_1 + \kappa_2 + 2$, one has
\begin{equation*}
  \begin{split}
      J(y_1, y_2) = &\left( \frac{1}{p} - \frac{1}{\kappa_1 + \kappa_2 + 2} \right )\int_{\Omega}\vert \nabla y_1 \vert^pdz + \left(\frac{1}{q} - \frac{1}{\kappa_1 + \kappa_2 + 2}\right )\|\nabla y_1 \|_{q,\eta}^{q}\\
      & + \left(\frac{1}{p} - \frac{1}{\kappa_1 + \kappa_2 + 2}\right )\int_{\Omega}\vert \nabla y_2 \vert^pdz + \left(\frac{1}{q} - \frac{1}{\kappa_1 + \kappa_2 + 2}\right )\|\nabla y_2 \|_{q,\eta}^{q}\\
      & + \left(\frac{1}{\kappa_1 + \kappa_2 + 2} - \frac{1}{1 - \nu}\right )\int_{\Omega}\left[a_1\vert y_1 \vert^{1 - \nu} + a_2\vert y_2 \vert^{1 - \nu}\right] dz\\
      \geq & \left( \frac{1}{q} - \frac{1}{\kappa_1 + \kappa_2 + 2} \right )\left[\rho_{\mathcal{H}}(\nabla y_1)+\rho_{\mathcal{H}}(\nabla y_2)\right] \\
       &+ \left(\frac{1}{\kappa_1 + \kappa_2 + 2} - \frac{1}{1 - \nu}\right )\int_{\Omega}\left[a_1\vert y_1 \vert^{1 - \nu} + a_2\vert y_2 \vert^{1 - \nu}\right] dz\\
  \end{split}
\end{equation*}

The following three cases are discussed.

\begin{itemize}[itemindent = 3.5em]
\item [Case (1):] $\|\nabla y_1\|_{\mathcal{H}}\rightarrow \infty$, $\|\nabla y_2\|_{\mathcal{H}}$ bounded.
\end{itemize}

By the use of Lemma \ref{prop2.2(liu)}$(iii)$, H\"{o}lder inequality, Poincar\'{e}'s inequality and Lemma \ref{prop2.2(Aro)}$(i)$, one has
\begin{equation*}
  \begin{split}
      J(y_1, y_2)
\geq& \left( \frac{1}{q} - \frac{1}{\kappa_1 + \kappa_2 + 2} \right )\rho_{\mathcal{H}}(\nabla y_1)- C_1\|y_1\|_{{1, p}}^{1 - \nu} - C_2\|y_2\|_{{1, p}}^{1 - \nu}\\
      \geq& C\|\nabla y_1\|_{\mathcal{H}}^{p} - C_1\|y_1\|_{{1, p}}^{1 - \nu} - C_2\|y_2\|_{{1, p}}^{1 - \nu}\\
       \geq &C\|\nabla y_1\|_{\mathcal{H}}^{p} - C_3\|\nabla y_1\|_{\mathcal{H}}^{1 - \nu} - C_4\|\nabla y_2\|_{\mathcal{H}}^{1 - \nu},
       \end{split}
\end{equation*}
where $C, C_1, C_2,C_3,C_4$ are positive constants. It is worth noting that $C_1, C_2$ comes from H\"{o}lder inequality and Poincar\'{e}'s inequality. Since $p > 1 - \nu>0$, we have $J(y_1, y_2)\rightarrow \infty$ as $\|\nabla y_1\|_{\mathcal{H}} \rightarrow \infty$.

\begin{itemize}[itemindent = 3.5em]
\item [Case (2):] $\|\nabla y_1\|_{\mathcal{H}}$ bounded, $\|\nabla y_2\|_{\mathcal{H}}\rightarrow  \infty$.
\end{itemize}
\begin{equation*}
  \begin{split}
      J(y_1, y_2) \geq & \left( \frac{1}{q} - \frac{1}{\kappa_1 + \kappa_2 + 2} \right )\rho_{\mathcal{H}}(\nabla y_2)\\
        &+ \left(\frac{1}{\kappa_1 + \kappa_2 + 2} - \frac{1}{1 - \nu}\right )\int_{\Omega}\left[a_1\vert y_1 \vert^{1 - \nu} + a_2\vert y_2 \vert^{1 - \nu}\right] dz\\
       \geq &C\|\nabla y_2\|_{\mathcal{H}}^{p} - C_3\|\nabla y_1\|_{\mathcal{H}}^{1 - \nu} - C_4\|\nabla y_2\|_{\mathcal{H}}^{1 - \nu},
  \end{split}
\end{equation*}
 Thus, we have $J(y_1, y_2)\rightarrow \infty$ as $\|\nabla y_2\|_{\mathcal{H}} \rightarrow \infty$.

\begin{itemize}[itemindent = 3.5em]
\item [Case (3):] $\|\nabla y_1\|_{\mathcal{H}}\rightarrow \infty$, $\|\nabla y_2\|_{\mathcal{H}} \rightarrow \infty$.
\end{itemize}
\begin{equation*}
  \begin{split}
      J(y_1, y_2) \geq & \left( \frac{1}{q} - \frac{1}{\kappa_1 + \kappa_2 + 2} \right )\left[\rho_{\mathcal{H}}(\nabla y_1)+\rho_{\mathcal{H}}(\nabla y_2)\right]\\
        &+ \left(\frac{1}{\kappa_1 + \kappa_2 + 2} - \frac{1}{1 - \nu}\right )\int_{\Omega}\left[a_1\vert y_1 \vert^{1 - \nu} + a_2\vert y_2 \vert^{1 - \nu}\right] dz\\
       \geq &C\|\nabla y_1\|_{\mathcal{H}}^{p} + C\|\nabla y_2\|_{\mathcal{H}}^{p} - C_3\|\nabla y_1\|_{\mathcal{H}}^{1 - \nu} - C_4\|\nabla y_2\|_{\mathcal{H}}^{1 - \nu},
  \end{split}
\end{equation*}
Thus, we have $J(y_1, y_2)\rightarrow \infty$ as $\|\nabla y_1\|_{\mathcal{H}}\rightarrow \infty$ and $\|\nabla y_2\|_{\mathcal{H}} \rightarrow \infty$.

Again since the definition of  $\|(y_1,y_2)\|$, we know that $\|(y_1,y_2)\|\rightarrow \infty$ if and only if $\|\nabla y_1\|_{\mathcal{H}}\rightarrow \infty$ or $\|\nabla y_2\|_{\mathcal{H}}\rightarrow \infty$, so the functional $J(y_1, y_2)$ is coercive on $\mathcal{N}_{\lambda}$.
\end{proof}
\begin{lemma}\label{3.2}
Suppose the assumptions $(H_1)$, $(H_2)$, $(H_{3})$ hold, then there exists $\lambda_{0} > 0$ such that $\mathcal{N}_{\lambda}^{0} = \emptyset$ for any $\lambda\in(0, \lambda_{0}$).
\end{lemma}
\begin{proof}
If $\lambda > 0$ such that $\mathcal{N}_{\lambda}^{0}\neq\emptyset$, then, by the definition of $\mathcal{N}_{\lambda}$, for any $(y_1, y_2)\in \mathcal{N}_{\lambda}^{0}\subset \mathcal{N}_{\lambda}$, one has
\begin{equation}\label{uv21}
    \begin{split}
\lambda(\kappa_1 + \kappa_2 + 2)&(\kappa_1 + \kappa_2 + 1)\int_{\Omega}\vert y_1\vert^{\kappa_1 + 1}\vert y_2\vert^{\kappa_2 + 1}dz\\
=& (\kappa_1 + \kappa_2 + 1)\| (y_1, y_2)\|_{1,p}+ (\kappa_1 + \kappa_2 + 1)\| (\nabla y_1, \nabla y_2)\|_{q,\eta}\\
&- (\kappa_1 + \kappa_2 + 1)\int_{\Omega}\left[a_1\vert y_1 \vert^{1 - \nu} + a_2\vert y_2 \vert^{1 - \nu}\right]dz.
    \end{split}
\end{equation}
Combing with \eqref{uv21} and the definition of $\mathcal{N}_{\lambda}^{0}$, there holds
\begin{equation}\label{uv < uvgamma}
 \begin{split}
   0 = &(\kappa_1 + \kappa_2 + 2 - p)\|  (y_1, y_2) \|_{1,p} +  (\kappa_1 + \kappa_2 + 2 - q)\|  (\nabla {y_1}, \nabla y_2) \|_{q,\eta}\\
   &- (\kappa_1 + \kappa_2 + \nu + 1)\int_{\Omega}\left[a_1\vert y_1 \vert^{1 - \nu} + a_2\vert y_2 \vert^{1 - \nu}\right]dz.
 \end{split}
\end{equation}
By using H\"{o}lder inequality, Poincar\'{e}'s inequality, assumption $(H_2)$, and the fact that the function $t^{\frac{1 - \nu}{p}}$ is concave about $t$, there is
\begin{equation}\label{gamma<1,h,0}
   \int_{\Omega}\left[a_1\vert y_1 \vert^{1 - \nu} + a_2\vert y_2 \vert^{1 - \nu}\right]dz\leq C_1\|y_1\|_{{1, p}}^{1 - \nu} + C_2\|y_2\|_{{1, p}}^{1 - \nu}\leq C_4 \|(y_1, y_2)\|_{1,p}^{\frac{1 - \nu}{p}},
\end{equation}
which combining \eqref{uv < uvgamma} yields
\begin{equation}\label{4.23}
   \|(y_1, y_2)\|_{1,p} \leq \left(\frac{(\kappa_1 + \kappa_2 + \nu + 1)C_4}{\kappa_1 + \kappa_2 + 2 - p} \right)^{\frac{p}{p+\gamma-1}}:=C_5.
\end{equation}

Again since $(y_1, y_2)\in \mathcal{N}_{\lambda}^{0}$, and $(H_1), (H_2), (H_{3})$ hold, one has
\begin{equation*}
  (p - 1)\|  (y_1, y_2) \|_{1,p}  \leq \lambda(\kappa_1 + \kappa_2 + 2)(\kappa_1 + \kappa_2 + 1)\int_{\Omega}\vert y_1\vert^{\kappa_1 + 1}\vert y_2\vert^{\kappa_2 + 1}dz,
\end{equation*}
which means
\begin{equation}\label{uvualpha + 1vbeta + 1}
 (p - 1) \| (y_1, y_2) \|_{1,p} \leq \lambda C_{6} \| y_1\|_{1, p}^{\kappa_1 + 1}\| y_2\|_{1, p}^{\kappa_2 + 1}\leq \lambda C_{7}\| (y_1, y_2) \|_{1,p}^{\frac{\kappa_1 + \kappa_2 + 2}{p}},
\end{equation}
where $C_{6}$, $C_{7}$ are two positive constants. In fact, by assumption $(H_{3})$, we know $\kappa_1 + \kappa_2 + 2 < p^*$, hence $\frac{p^*}{\kappa_1 + 1} - \frac{p^*}{p^* - (\kappa_2 + 1)} > 0$. Given $ \epsilon_{0}$ such that
\begin{equation}\label{epsilon0}
  0 < \epsilon_{0} < \frac{p^*}{\kappa_1 + 1} - \frac{p^*}{p^* - (\kappa_2 + 1)},
\end{equation}
let
\begin{equation*}
  m_1: = p^* - \epsilon_{0}(\kappa_1 + 1), \ \ m_2: = \frac{[p^* - \epsilon_{0}(\kappa_1 + 1)](\kappa_2 + 1)}{p^* - (\epsilon_{0} + 1)(\kappa_1 + 1)}.
\end{equation*}
Then, $m_1\in(\kappa_1 + 1, p^*), $ $m_2\in(\kappa_2 + 1, p^*)$.
By using H\"{o}lder inequality and Poincar\'{e}'s inequality, we have
\begin{equation}\label{uvholder}
   \begin{split}
      \int_{\Omega}\vert y_1\vert^{\kappa_1 + 1}\vert y_2\vert^{\kappa_2 + 1}dz \leq& C_8 \| y_1\|_{m_1}^{\kappa_1 + 1}\| y_2\|_{m_2}^{\kappa_2 + 1}\\
      \leq&  C_{9} \|y_1\|_{1, p}^{\kappa_1 + 1}\|y_2\|_{1, p}^{\kappa_2 + 1},
   \end{split}
\end{equation}
here, $C_8,C_{9}$ are two positive constants. By Young's inequality and the properties of concave function $t\rightarrow t^{1/p}$, one has,
\begin{equation}\label{young}
   \begin{split}
      \|y_1\|_{1, p}^{\frac{\kappa_1 + 1}{\kappa_1 + \kappa_2 + 2}}\|y_2\|_{1, p}^{\frac{\kappa_2 + 1}{\kappa_1 + \kappa_2 + 2}}&\leq\frac{(\kappa_1 + 1)\|y_1\|_{1, p}}{\kappa_1 + \kappa_2 + 2} + \frac{(\kappa_2 + 1)\|y_2\|_{1, p}}{\kappa_1 + \kappa_2 + 2}\\
       &\leq  \|y_1\|_{1, p} + \|y_2\|_{1, p}\leq C_{10} \|(y_1, y_2)\|_{1,p}^{\frac{1}{p}},
   \end{split}
\end{equation}
here $C_{10}$ is a positive constant. From \eqref{uvholder} and \eqref{young} we know that \eqref{uvualpha + 1vbeta + 1} holds.

From \eqref{uvualpha + 1vbeta + 1} we have
\begin{equation}\label{18}
  \lambda\geq(p-1)C_7^{-1}\|(y_1,y_2)\|_{1,p}^{1-\frac{\kappa_1+\kappa_2+2}{p}}=(p-1)C_7^{-1}\|(y_1,y_2)\|_{1,p}^{\frac{p-(\kappa_1+\kappa_2+2)}{p}}.
\end{equation}
Now, $\frac{p-(\kappa_1+\kappa_2+2)}{p}<0$ and by \eqref{4.23}, one has $\|(y_1,y_2)\|_{1,p}\leq C_5$. Thus, \eqref{18} comes to
\begin{equation*}
  \lambda\geq (p-1)C_7^{-1}C_5^{-\frac{p-(\kappa_1+\kappa_2+2)}{p}}.
\end{equation*}
Then the lemma holds with $\lambda_0:=(p-1)C_7^{-1}C_5^{-\frac{p-(\kappa_1+\kappa_2+2)}{p}}$.
 %we know when $\lambda\rightarrow 0^{ + }$, $\| (y_1, y_2) \|_{1,p}\rightarrow  + \infty$, combining with \eqref{4.23}, we have a contradiction. So we can state that there exists $\lambda_{0} > 0$ such that for $\lambda\in(0, \lambda_{0}$), the set $\mathcal{N}_{\lambda}^{0} = \emptyset$.
\end{proof}

\begin{lemma}\label{jn +  +  < 0}
Suppose the assumptions $(H_1)$, $(H_2)$ hold and $\mathcal{N}_{\lambda}^{ + }\neq \emptyset$, then for all $(y_1, y_2)\in\mathcal{N}_{\lambda}^{ + }$, $J(y_1, y_2) < 0$.
\end{lemma}
\begin{proof}
For $(y_1, y_2)\in\mathcal{N}_{\lambda}^{ + }$, by the definition of $\mathcal{N}_{\lambda}^ + $, there holds
\begin{equation}\label{n +  + u}
  \begin{split}
    \lambda(\kappa_1 + \kappa_2 + 2)&(\kappa_1 + \kappa_2 + 1)\int_{\Omega}\vert y_1\vert^{\kappa_1 + 1}\vert y_2\vert^{\kappa_2 + 1}dz\\
   \leq & (p - 1) \| (y_1, y_2)\|_{1,p} + (q - 1)\| (\nabla y_1, \nabla y_2)\|_{q,\eta} \\
   & + \nu\int_{\Omega}\left[a_1\vert y_1 \vert^{1 - \nu} + a_2\vert y_2 \vert^{1 - \nu}\right]dz.
    \end{split}
\end{equation}
Again since $\mathcal{N}_{\lambda}^{ + }\subset\mathcal{N}_{\lambda}$, from the definition of $\mathcal{N}_{\lambda}$, we have
\begin{equation}\label{nuau}
  \begin{split}
   \nu\int_{\Omega}\left[a_1\vert y_1 \vert^{1 - \nu} + a_2\vert y_2 \vert^{1 - \nu}\right]dz =& \nu\| (y_1, y_2)\|_{1,p} + \nu\| (\nabla y_1, \nabla y_2)\|_{q,\eta}\\
   & - \lambda\nu(\kappa_1 + \kappa_2 + 2)\int_{\Omega}\vert y_1\vert^{\kappa_1 + 1}\vert y_2\vert^{\kappa_2 + 1}dz.
     \end{split}
\end{equation}
Using \eqref{n +  + u} in \eqref{nuau}, we have
\begin{equation}\label{lambdauv1}
  \begin{split}
   \lambda\int_{\Omega}\vert y_1\vert^{\kappa_1 + 1}\vert y_2\vert^{\kappa_2 + 1}dz\leq&\frac{p - 1 + \nu}{(\kappa_1 + \kappa_2 + 2)(\kappa_1 + \kappa_2 + 1 + \nu)} \| (y_1, y_2)\|_{1,p}\\
   & + \frac{q - 1 + \nu}{(\kappa_1 + \kappa_2 + 2)(\kappa_1 + \kappa_2 + 1 + \nu)}\| (\nabla y_1, \nabla y_2)\|_{q,\eta}.
   \end{split}
\end{equation}
Combining the definition of the functional $J(y_1, y_2)$ with \eqref{nuau}, \eqref{lambdauv1}, for all $(y_1, y_2)\in \mathcal{N}_{\lambda}^{ + }$,
\begin{equation*}
   \begin{split}
      J(y_1, y_2) = &\left( \frac{1}{p} - \frac{1}{1 - \nu} \right )\| (y_1, y_2) \|_{1,p} + \left(\frac{1}{q} - \frac{1}{1 - \nu}\right )\| (\nabla y_1, \nabla y_2)\|_{q,\eta}\\
      &  +\lambda \left[\frac{(\kappa_1 + \kappa_2 + 2)}{1 - \nu} - 1 \right ]\int_{\Omega}\vert y_1\vert^{\kappa_1 + 1}\vert y_2\vert^{\kappa_2 + 1}dz\\
      \leq &\left( \frac{1}{p} - \frac{1}{1 - \nu} \right )\| (y_1, y_2) \|_{1,p} + \left(\frac{1}{q} - \frac{1}{1 - \nu}\right )\|(\nabla {y_1}, \nabla y_2)\|_{q,\eta}\\
      & + \frac{1}{1 - \nu}\frac{p - 1 + \nu}{\kappa_1 + \kappa_2 + 2} \| (y_1, y_2)\|_{1,p} + \frac{1}{1 - \nu}\frac{q - 1 + \nu}{\kappa_1 + \kappa_2 + 2}\|(\nabla {y_1}, \nabla y_2)\|_{q,\eta}\\
       = &\frac{p - 1 + \nu}{1 - \nu}\left( \frac{1}{\kappa_1 + \kappa_2 + 2} - \frac{1}{p}\right)\| (y_1, y_2)\|_{1,p}\\
      & + \frac{q - 1 + \nu}{1 - \nu}\left( \frac{1}{\kappa_1 + \kappa_2 + 2} - \frac{1}{q}\right)\|(\nabla {y_1}, \nabla y_2)\|_{q,\eta}.
    \end{split}
\end{equation*}
Since $1 < p < q < \kappa_1 + \kappa_2 + 2$, we have $J(y_1, y_2) < 0$. Which means $\inf_{\mathcal{N}_{\lambda}^{ + }}J(y_1, y_2) < 0$.
\end{proof}

\begin{lemma}\label{lem3.4}
Suppose the assumptions $(H_1)$, $(H_2)$, $(H_{3})$ hold. If $\mathcal{N}_{\lambda}^{ - }\neq \emptyset$ for some $\lambda>0$, then there exists $\lambda_1 > 0$ such that for all $\lambda\in(0, \lambda_1$) and $(y_1, y_2)\in\mathcal{N}_{\lambda}^{ - }$, $J(y_1, y_2) > 0$.
\end{lemma}
\begin{proof}
On the one hand, for $(y_1, y_2)\in\mathcal{N}_{\lambda}^{ - }$, we have
\begin{equation}\label{n--u}
 \begin{split}
  (p - 1) \| (y_1, y_2)\|_{1,p} < &  \lambda(\kappa_1 + \kappa_2 + 2)(\kappa_1 + \kappa_2 + 1)\int_{\Omega}\vert y_1\vert^{\kappa_1 + 1}\vert y_2\vert^{\kappa_2 + 1}dz\\
  & - (q - 1)\| (\nabla y_1, \nabla y_2)\|_{q,\eta} \\
  &  - \nu\int_{\Omega}\left[a_1\vert y_1 \vert^{1 - \nu} + a_2\vert y_2 \vert^{1 - \nu}\right]dz.
 \end{split}
\end{equation}
Since $\nu > 0$, $q > 1$, and $(H_1)$, $(H_2)$ hold, one has
\begin{equation}\label{uv < ualpha + 1vbeta + 1}
   (p - 1) \| (y_1, y_2)\|_{1,p} <  \lambda(\kappa_1 + \kappa_2 + 2)(\kappa_1 + \kappa_2 + 1)\int_{\Omega}\vert y_1\vert^{\kappa_1 + 1}\vert y_2\vert^{\kappa_2 + 1}dz.
\end{equation}
Combining with \eqref{uvualpha + 1vbeta + 1}, we can get for some $C_{11}>0$, \eqref{uv < ualpha + 1vbeta + 1} comes to
\begin{equation}\label{lambad1}
   \| (y_1, y_2)\|_{1,p} < \lambda C_{11}\| (y_1, y_2) \|_{1,p}^{\frac{\kappa_1 + \kappa_2 + 2}{p}}.
\end{equation}

On the other hand, suppose there exists a point $(y_1, y_2)\in\mathcal{N}_{\lambda}^{ - }$ and $J(y_1, y_2)\leq 0$, i.e.,
\begin{equation}\label{n - juv < 0}
  \begin{split}
  J(y_1, y_2) =  &\frac{1}{p}\int_{\Omega}(\vert \nabla {y_1}\vert^p + \vert \nabla y_2\vert^p)dz  + \frac{1}{q}\int_{\Omega}\eta\vert \nabla {y_1}\vert^{q}dz + \frac{1}{q}\int_{\Omega}\eta\vert \nabla y_2\vert^{q}dz\\
   &  -  \frac{1}{1 - \nu}\int_{\Omega}\left[a_1\vert y_1 \vert^{1 - \nu} + a_2\vert y_2 \vert^{1 - \nu}\right]dz  - \lambda\int_{\Omega}\vert y_1\vert^{\kappa_1 + 1} \vert y_2\vert^{\kappa_2 + 1}dz\leq 0.
 \end{split}
\end{equation}
Again since $ (y_1, y_2)\in\mathcal{N}_{\lambda}^{ - }\subset \mathcal{N}_{\lambda}$, from the definition of $\mathcal{N}_{\lambda}$, one has
\begin{equation}\label{uvn - lambda}
  \begin{split}
      - \lambda\int_{\Omega}\vert y_1\vert^{\kappa_1 + 1}\vert y_2\vert^{\kappa_2 + 1}dz = & - \frac{1}{\kappa_1 + \kappa_2 + 2}\| (y_1, y_2)\|_{1,p} - \frac{1}{\kappa_1 + \kappa_2 + 2}\|(\nabla {y_1}, \nabla y_2)\|_{q,\eta} \\
   & + \frac{1}{\kappa_1 + \kappa_2 + 2}\int_{\Omega}\left[a_1\vert y_1 \vert^{1 - \nu} + a_2\vert y_2 \vert^{1 - \nu}\right]dz.
   \end{split}
\end{equation}
Using \eqref{uvn - lambda} in \eqref{n - juv < 0} to get
\begin{equation*}
 \begin{split}
&\left(\frac{1}{p} - \frac{1}{\kappa_1 + \kappa_2 + 2}\right)\| (y_1, y_2)\|_{1,p} + \left( \frac{1}{q} - \frac{1}{\kappa_1 + \kappa_2 + 2}\right)\|(\nabla {y_1}, \nabla y_2)\|_{q,\eta} \\
    & + \left(\frac{1}{\kappa_1 + \kappa_2 + 2}  -  \frac{1}{1 - \nu}\right)\int_{\Omega}\left[a_1\vert y_1 \vert^{1 - \nu} + a_2\vert y_2 \vert^{1 - \nu}\right]dz\leq0.
 \end{split}
\end{equation*}
Taking into account that $q < \kappa_1 + \kappa_2 + 2$ and $(H_1)$ holds, we have
\begin{equation*}
  \begin{split}
   &\| (y_1, y_2)\|_{1,p}\leq\frac{p(\kappa_1 + \kappa_2 + \nu + 1)}{(\kappa_1 + \kappa_2 + 2 - p)(1 - \nu)}\int_{\Omega}\left[a_1\vert y_1 \vert^{1 - \nu} + a_2\vert y_2 \vert^{1 - \nu}\right]dz.
  \end{split}
\end{equation*}
Thus, using $(H_2)$, and proceeding as in the proof of \eqref{gamma<1,h,0}, for some positive cobstants $C_{12}, C_{13}$ and $C_{14}$, we have
\begin{equation}\label{juu < 0}
  \begin{split}
\| (y_1, y_2)\|_{1,p}&\leq C_{12}(\|y_1\|_p^{1 - \nu} + \|y_2\|_p^{1 - \nu})\\
&\leq C_{13}(\|y_1\|_{1, p}^{1 - \nu} + \|y_2\|_{1, p}^{1 - \nu})\\
&\leq C_{14}\|(y_1, y_2)\|_{1,p}^{\frac{1 - \nu}{p}},
  \end{split}
\end{equation}
here we have used the Poincar\'{e}'s inequality and that the function $t \rightarrow t^{\frac{1-\nu}{p}}$ is concave.
Combining \eqref{lambad1} and \eqref{juu < 0}, we have
\begin{equation*}
\left(\frac{1}{\lambda C_{11}}\right)^{\frac{1}{\kappa_1 + \kappa_2 + 2 - p}} < C_{14}^{\frac{1}{p + \nu - 1}}.
\end{equation*}
Thus, since $1<p<\kappa_1+\kappa_2+2$, we get
\begin{equation*}
  \lambda>C_{11}^{-1}C_{14}^{-\frac{\kappa_1+\kappa_2+2-p}{p-1+\gamma}}.
\end{equation*}
Then the lemma holds with $\lambda_1:=C_{11}^{-1}C_{14}^{-\frac{\kappa_1+\kappa_2+2-p}{p-1+\gamma}}$.
%let $\lambda\rightarrow 0^{ + }$, one has
%\begin{equation*}
%\frac{1}{ C_{12}} < \lambda C_{15}^{\frac{\kappa_1 + \kappa_2 + 2 - p}{p + \nu - 1}} = 0.
%\end{equation*}
%This is a contradiction. Thus we can find a $\lambda_1 > 0$ such that for all $\lambda\in(0, \lambda_1$) and $(y_1, y_2)\in\mathcal{N}_{\lambda}^{ - }$, there holds $J(y_1, y_2) > 0$.
\end{proof}

\begin{lemma}\label{t1t2}
 Suppose the assumptions $(H_1)$, $(H_2)$, $(H_{3})$ hold and $(y_1, y_2)\in W_{0}^{1, \mathcal{H}}(\Omega) \times W_{0}^{1, \mathcal{H}}(\Omega)$ with $\int_{\Omega}\vert y_1\vert^{\kappa_1 + 1}\vert y_2\vert^{\kappa_2 + 1}dz > 0$, then there exists $ \tilde{\lambda}_{0} > 0$ such that for $\lambda\in (0, \tilde{\lambda}_{0}), $ there exist two positive constants $t_1, t_2$ such that $t_1<t_2$, $t_1(y_1, y_2)\in \mathcal{N}_{\lambda}^{ + }$, $t_2(y_1, y_2)\in \mathcal{N}_{\lambda}^{ - }$.
\end{lemma}
\begin{proof}
For a given $(y_1, y_2)\in W_{0}^{1, \mathcal{H}}(\Omega) \times W_{0}^{1, \mathcal{H}}(\Omega)$ with $\int_{\Omega}\vert y_1\vert^{\kappa_1 + 1}\vert y_2\vert^{\kappa_2 + 1}dz > 0$, define the function $\tilde{F}_{(y_1,y_2)}(t):(0,  + \infty)\rightarrow \mathbb{R}$ by
\begin{equation*}
  \tilde{F}_{(y_1, y_2)}(t): = t^{p - 1 - (\kappa_1 + \kappa_2 + 1)}\| (y_1, y_2)\|_{1,p} - t^{ - \nu - (\kappa_1 + \kappa_2 + 1)}\int_{\Omega}\left[a_1\vert y_1 \vert^{1 - \nu} + a_2\vert y_2 \vert^{1 - \nu}\right]dz.
\end{equation*}
\par Now we discuss the geometry of $\tilde{F}_{(y_1,y_2)}(t)$.
If $\tilde{F}_{(y_1,y_2)}^{\prime}(t) = 0$, then
\begin{equation*}
\begin{split}
(p - \kappa_1 - \kappa_2 - 2)&t^{p - \kappa_1 - \kappa_2 - 3}\| (y_1, y_2)\|_{1,p}\\
&= ( - \nu - \kappa_1 - \kappa_2 - 1)t^{ - \nu - \kappa_1 - \kappa_2 - 2}\int_{\Omega}\left[a_1\vert y_1 \vert^{1 - \nu} + a_2\vert y_2 \vert^{1 - \nu}\right]dz.
  \end{split}
\end{equation*}
Then the unique stationary point, denoted by $\tilde{t}_{0}$, of the function $\tilde{F}_{(y_1,y_2)}(t)$ can be solved as
\begin{equation*}
  \tilde{t}_{0} = \left[ \frac{(\nu + \kappa_1 + \kappa_2 + 1)\int_{\Omega}\left[a_1\vert y_1 \vert^{1 - \nu} + a_2\vert y_2 \vert^{1 - \nu}\right]dz}{(\kappa_1 + \kappa_2 + 2 - p)\| (y_1, y_2)\|_{1,p}}\right]^{\frac{1}{p + \nu - 1}},
\end{equation*}
and so
\begin{equation}\label{phiu}
\begin{split}
&\tilde{F}_{(y_1,y_2)}(\tilde{t}_{0}) \\
&=\frac{p + \nu - 1}{\kappa_1 + \kappa_2 + 2 - p}\left[ \frac{\kappa_1 + \kappa_2 + 2 - p}{\nu + \kappa_1 + \kappa_2 + 1} \right]^{\frac{\nu + \kappa_1 + \kappa_2 + 1}{p + \nu - 1}}\frac{(\| (y_1, y_2)\|_{1,p})^{\frac{\nu + \kappa_1 + \kappa_2 + 1}{p + \nu - 1}}}{\left\{\int_{\Omega}\left[a_1\vert y_1 \vert^{1 - \nu} + a_2\vert y_2 \vert^{1 - \nu}\right]dz\right\}^{\frac{\kappa_1 + \kappa_2 + 2 - p}{p + \nu - 1}}}.
\end{split}
\end{equation}
Exactly as in the proof of the inequality \eqref{gamma<1,h,0} (and also of the inequality \eqref{uvualpha + 1vbeta + 1}), by using $(H_2)$, the Poincar\'{e}'s inequality and H\"{o}lder inequalities, and the fact that $t\rightarrow t^{(1-\nu)/p}$ is a concave function, we get
\begin{equation}\label{holderaxu}
  \int_{\Omega}\left[a_1\vert y_1 \vert^{1 - \nu} + a_2\vert y_2 \vert^{1 - \nu}\right]dz\leq C_{15}\|(y_1, y_2)\|_{1,p}^{\frac{1 - \nu}{p}},
\end{equation}
\begin{equation}\label{holderlambaduv}
  \lambda(\kappa_1 + \kappa_2 + 2)\int_{\Omega}\vert y_1\vert^{\kappa_1 + 1}\vert y_2\vert^{\kappa_2 + 1}dz\leq \frac{\lambda C_7}{\kappa_1 + \kappa_2 + 1} \| (y_1, y_2) \|_{1,p}^{\frac{\kappa_1 + \kappa_2 + 2}{p}}.
\end{equation}
Combining \eqref{phiu} with \eqref{holderaxu} and \eqref{holderlambaduv}, we have for some
positive constants $C_{16}$ and $C_{17}$,
\begin{equation}
\begin{split}
\tilde{F}_{(y_1, y_2)}(\tilde{t}_{0}) &- \lambda(\kappa_1 + \kappa_2 + 2)\int_{\Omega}\vert y_1\vert^{\kappa_1 + 1}\vert y_2\vert^{\kappa_2 + 1}dz\\
\geq&\frac{p + \nu - 1}{\kappa_1 + \kappa_2 + 2 - p}\left[ \frac{\kappa_1 + \kappa_2 + 2 - p}{\nu + \kappa_1 + \kappa_2 + 1} \right]^{\frac{\nu + \kappa_1 + \kappa_2 + 1}{p + \nu - 1}}\frac{(\|(y_1, y_2)\|_{1,p})^{\frac{\nu + \kappa_1 + \kappa_2 + 1}{p + \nu - 1}}}{\left[C_{15}\|(y_1, y_2)\|_{1,p}^{\frac{1 - \nu}{p}}\right]^{\frac{\kappa_1 + \kappa_2 + 2 - p}{p + \nu - 1}}}\\
   &- \frac{\lambda C_7}{\kappa_1 + \kappa_2 + 1}\| (y_1, y_2) \|_{1,p}^{\frac{\kappa_1 + \kappa_2 + 2}{p}}\\
  =& (C_{16} - \lambda C_{17})\| (y_1, y_2) \|_{1,p}^{\frac{\kappa_1 + \kappa_2 + 2}{p}}.
\end{split}
\end{equation}
Set $\tilde{\lambda}_{0} = \frac{C_{16}}{C_{17}} > 0$ which independent of $(y_1,y_2)$ and $\lambda$, then for every $\lambda\in(0, \tilde{\lambda}_{0})$, one has
\begin{equation*}
  \tilde{F}_{(y_1, y_2)}(\tilde{t}_{0}) - \lambda(\kappa_1 + \kappa_2 + 2)\int_{\Omega}\vert y_1\vert^{\kappa_1 + 1}\vert y_2\vert^{\kappa_2 + 1}dx > 0.
\end{equation*}

\par Now we consider, for $\lambda\in(0,\tilde{\lambda}_0$), the function $ F_{(y_1, y_2)}:(0,\infty)\rightarrow \mathbb{R}$ defined by
\begin{equation*}
  F_{(y_1, y_2)}(t): = \tilde{F}_{(y_1, y_2)}(t) +  t^{q - 1 - (\kappa_1 + \kappa_2 + 1)}\|(\nabla {y_1},\nabla y_2) \|_{q,\eta}.
\end{equation*}
Taking into account $(H_1)$ holds and the fact that $\nu + \kappa_1 + \kappa_2 + 1 > \kappa_1 + \kappa_2 + 2 - p > \kappa_1 + \kappa_2 + 2 - q > 0$, we have
\begin{equation*}
\begin{split}
    F_{(y_1, y_2)}(t)\geq \tilde{F}_{(y_1, y_2)}(t), \ \ \forall\  t\in(0,  + \infty),
\end{split}
\end{equation*}
and
\begin{equation*}
  \lim_{t\rightarrow0^{ + }} F_{(y_1, y_2)}(t) = \lim_{t\rightarrow 0^{ + }}\tilde{F}_{(y_1, y_2)}(t) =  - \infty,
\end{equation*}
\begin{equation*}
      \lim_{t\rightarrow  + \infty} F_{(y_1, y_2)}(t) = \lim_{t\rightarrow  + \infty}\tilde{F}_{(y_1, y_2)}(t) = 0.
\end{equation*}
It is clearly that, for any $t>0$,
\begin{equation}\label{psiandphitu}
 \psi_{(y_1, y_2)}^{\prime}(t) = t^{\kappa_1 + \kappa_2 + 1}\left[F_{(y_1, y_2)}(t) - \lambda(\kappa_1 + \kappa_2 + 2)\int_{\Omega}\vert y_1\vert^{\kappa_1 + 1}\vert y_2\vert^{\kappa_2 + 1}dz\right],
\end{equation}
\begin{equation}\label{psiandphitu2}
\begin{split}
  \psi_{(y_1,y_2)}^{\prime\prime}(t) =& (\kappa_1 + \kappa_2 + 1)t^{\kappa_1 + \kappa_2}\left[F_{(y_1, y_2)}(t) - \lambda(\kappa_1 + \kappa_2 + 2)\int_{\Omega}\vert y_1\vert^{\kappa_1 + 1}\vert y_2\vert^{\kappa_2 + 1}dz\right]\\
   & + t^{\kappa_1 + \kappa_2 + 1}F_{(y_1, y_2)}^{\prime}(t).
\end{split}
\end{equation}
Thus, combining \eqref{psiandphitu} and \eqref{psiandphitu2} with Lemma \ref{lemma2.3}, it follows that if $t > 0$ satisfies $F_{(y_1, y_2)}(t) = \lambda(\kappa_1 + \kappa_2 + 2)\int_{\Omega}\vert y_1\vert^{\kappa_1 + 1}\vert y_2\vert^{\kappa_2 + 1}dz$ and $F_{(y_1, y_2)}^{\prime}(t) > 0$ (respctively $F_{(y_1, y_2)}^{\prime}(t) < 0$),  then $t(y_1, y_2)\in\mathcal{N}_{\lambda}^{ + }$(resp. $t(y_1, y_2)\in\mathcal{N}_{\lambda}^{ - })$.

Since $\lim_{t\rightarrow0^{ + }} F_{(y_1, y_2)}(t) =  - \infty$, $\lim_{t\rightarrow  + \infty} F_{(y_1, y_2)}(t) =0$ and $F_{(y_1, y_2)}(\tilde{t}_{0})\geq \tilde{F}_{(y_1, y_2)}(\tilde{t}_{0})>0$, it follows that there exists a point $t_0\in(0,+\infty)$ where $F_{(y_1, y_2)}(t) $ attains its maximum value. Such a $t_0$ is clearly a critical point of $F_{(y_1, y_2)}$(that is, $F_{(y_1, y_2)}^{\prime}(t_0)=0$ ). Now we prove that $t_0$ is the unique critical point of $F_{(y_1, y_2)}(t)$. Define $ F_1:(0,\infty)\rightarrow \mathbb{R}$ by
\begin{equation}\label{Phi1definition}
\begin{split}
 F_1(t): & = t^{\nu + \kappa_1 + \kappa_2 + 2} F_{(y_1, y_2)}^{\prime}(t) + ( - \nu - \kappa_1 - \kappa_2 - 1)\int_{\Omega}\left[a_1\vert y_1 \vert^{1 - \nu} + a_2\vert y_2 \vert^{1 - \nu}\right]dz.
\end{split}
\end{equation}
From the definition of $F_{(y_1, y_2)}^{\prime}(t)$, \eqref{Phi1definition} becomes
\begin{equation}\label{Phit}
\begin{split}
  F_1(t)  = &(p - \kappa_1 - \kappa_2 - 2)t^{p + \nu - 1}\| (y_1, y_2)\|_{1,p} + (q - \kappa_1 - \kappa_2 - 2)t^{q + \nu - 1}\| (\nabla y_1, \nabla y_2)\|_{q,\eta}
\end{split}
\end{equation}
Obviously, from the definition of $F_1(t)$, one has for any $t>0$,
\begin{equation*}
F_{(y_1, y_2)}^{\prime}(t) = 0\Leftrightarrow
F_1(t) -  ( - \nu - \kappa_1 - \kappa_2 - 1)\int_{\Omega}\left[a_1\vert y_1 \vert^{1 - \nu} + a_2\vert y_2 \vert^{1 - \nu}\right]dz = 0.
\end{equation*}
From \eqref{Phit} and $\nu < p < q < \kappa_1 + \kappa_2 + 2$, one has for any $t>0$,
\begin{equation*}
  F_1^{\prime}(t) < 0.
\end{equation*}
Then $ F_1$ is injective, and so, by the above equivalence, $t_0$ is the unique critical point of $F_{(y_1, y_2)}$. From this fact, and since $\lim_{t\rightarrow0^{ + }} F_{(y_1, y_2)}(t) =  - \infty, \lim_{t\rightarrow  + \infty} F_{(y_1, y_2)}(t) =0$, and
%Combining with the above analysis of the geometry of $F_{(y_1, y_2)}(t)$, we can know that there exists $t_{0} > 0$ such that $t_{0}$ are the unique extreme point of the function $F_{(y_1, y_2)}(t)$ (because of the injectivity of $F_1$). Which means
\begin{equation}\label{phivtphivsextremepoint}
  F_{(y_1, y_2)}(t_{0}) = \max \limits_{ t\in(0,\infty)} F_{(y_1, y_2)}(t),
\end{equation}
it follows that $ F_{(y_1, y_2)}^{\prime}(t)>0$ for any $t\in(0,t_0)$ and that $F_{(y_1, y_2)}^{\prime}(t)<0$ for any $t\in(t_0,\infty)$. Now,
\begin{equation}\label{phiut0 - lambdauv}
  F_{(y_1, y_2)}(t_{0}) \geq  F_{(y_1, y_2)}(\tilde{t}_{0}) \geq \tilde{F}_{(y_1, y_2)}(\tilde{t}_{0}) > \lambda(\kappa_1 + \kappa_2 + 2)\int_{\Omega}\vert y_1\vert^{\kappa_1 + 1}\vert y_2\vert^{\kappa_2 + 1}dz > 0,
\end{equation}
thus we can find unique $t_1 < t_{0}$ such that
\begin{equation*}
  F_{(y_1, y_2)}(t_1) = \lambda(\kappa_1 + \kappa_2 + 2)\int_{\Omega}\vert y_1\vert^{\kappa_1 + 1}\vert y_2\vert^{\kappa_2 + 1}dz, ~~~~F_{(y_1, y_2)}^{\prime}(t_1) > 0,
\end{equation*}
so, $t_1(y_1, y_2)\in \mathcal{N}_{\lambda}^{ + }$.

Similarly, there exists unique $t_2  >  t_{0}$ such that
\begin{equation*}
  F_{(y_1, y_2)}(t_2) = \lambda(\kappa_1 + \kappa_2 + 2)\int_{\Omega}\vert y_1\vert^{\kappa_1 + 1}\vert y_2\vert^{\kappa_2 + 1}dz, ~~~~ F_{(y_1, y_2)}^{\prime}(t_2) < 0,
\end{equation*}
thus $t_2(y_1, y_2)\in \mathcal{N}_{\lambda}^{ - }$.

Moreover, we can shown that $\psi_{(y_1,y_2)}(t)$ attains its maximum value at $t=t_2$ and its minimum value at $t=t_1$. In fact, since $F_{(y_1, y_2)}$ is strictly increasing on $(0,t_0)$ and strictly decreasing on $(t_0,\infty)$, from \eqref{psiandphitu} and $t_1<t_0<t_2$, one has the sign of $\psi_{(y_1,y_2)}^{\prime}(t)$ is constant on each one of the intervals $(0,t_1)$, $(t_1,t_2)$ and $(t_2,\infty)$. Again since $\psi_{(y_1,y_2)}^{\prime}(t)>0$ for $t\in(t_1,t_2)$  it follows that $\psi_{(y_1,y_2)}^{\prime}(t)<0$ for $t\in(0,t_1)$, and similar arguments give that $\psi_{(y_1,y_2)}^{\prime}(t)>0$ for all $t\in(t_1,t_2)$ and  $\psi_{(y_1,y_2)}^{\prime}(t)<0$ for all $t\in(t_2,\infty)$.
\end{proof}
\begin{lemma}\label{t3}
Suppose the assumptions $(H_1)$, $(H_2)$, $(H_{3})$ hold and $\lambda\in (0, \tilde{\lambda}_{0})$. If $(y_1, y_2)\in W_{0}^{1, \mathcal{H}}(\Omega) \times W_{0}^{1, \mathcal{H}}(\Omega)\backslash\{(0,0)\}$, then there exists a unique constant $t_{3} > 0$ such that $t_{3}(y_1, y_2)\in \mathcal{N}_{\lambda}^{ + }$.
\end{lemma}
\begin{proof}
If $\int_{\Omega}\vert y_1\vert^{\kappa_1 + 1}\vert y_2\vert^{\kappa_2 + 1}dz > 0$, then Lemma \ref{t1t2} gives the required $t_{3}$, and an inspection of the proof of Lemma \ref{t1t2} gives also that such a $t_3$ is unique.

If $\int_{\Omega}\vert y_1\vert^{\kappa_1 + 1}\vert y_2\vert^{\kappa_2 + 1}dz = 0$.
 % Since $(u, v)\neq(0, 0)$, we have
%\begin{equation*}
%  \int_{\Omega}\left(a_1\vert y_1 \vert^{1 - \nu} + a_2\vert y_2 \vert^{1 - \nu}\right)dx > 0,
%\end{equation*}
Let $(y_1,y_2)\in W_{0}^{1, \mathcal{H}}(\Omega) \times W_{0}^{1, \mathcal{H}}(\Omega)\backslash\{(0,0)\}$, define function $\tilde{F}_{(y_1, y_2)}(t):(0,\infty)\rightarrow \mathbb{R}$ by
\begin{equation*}
  \tilde{F}_{(y_1, y_2)}(t): = t^{p - 1}\| (y_1, y_2)\|_{1,p} - t^{ - \nu}\int_{\Omega}\left[a_1\vert y_1 \vert^{1 - \nu} + a_2\vert y_2 \vert^{1 - \nu}\right]dz.
\end{equation*}
Since $p - 1 > 0 >  - \nu$, then we have $\tilde{F}_{(y_1, y_2)}^{\prime}(t) > 0$ for any $t>0$,
\begin{equation*}
 \lim_{t\rightarrow 0^{ + }}\tilde{F}_{(y_1, y_2)}(t) =  - \infty, ~~~~ \lim_{t\rightarrow  + \infty}\tilde{F}_{(y_1, y_2)}(t) =  + \infty.
\end{equation*}
%So, there exists unique $\hat{t}_{3} > 0$ such that $ \tilde{F}_{(y_1, y_2)}(\hat{t}_{3}) = 0$.
Since
\begin{equation*}
\psi_{(y_1, y_2)}^{\prime}(t)  = \tilde{F}_{(y_1, y_2)}(t) +  t^{q - 1}\|(\nabla {y_1},\nabla y_2) \|_{q,\eta},
\end{equation*}
and $q > 1$, it follows that $\lim_{t\rightarrow 0^{ + }}\psi_{(y_1, y_2)}^{\prime}(t)=- \infty$, $\lim_{t\rightarrow  + \infty}\psi_{(y_1, y_2)}^{\prime}(t) =  + \infty$, and that $\psi_{(y_1, y_2)}^{\prime\prime}(t)>0$ for any $t>0$. Then there exists a unique $t_3>0$ such that $\psi_{(y_1, y_2)}^{\prime}(t_{3}) = 0$ and $\psi_{(y_1, y_2)}^{\prime\prime}(t_{3}) > 0$, and so, by Lemma \ref{lemma2.3}, there exists a unique $t_3>0$ such that $t_{3}(y_1, y_2)\in\mathcal{N}_{\lambda}^{ + }$.
\end{proof}

\vskip0.5cm
Let $\lambda^* = \min\{\lambda_{0}, \lambda_1, \tilde{\lambda}_{0}\}$.

\begin{lemma}\label{N--}
Suppose the assumptions $(H_1)$, $(H_2)$, $(H_{3})$ hold, then for $\lambda\in (0, \lambda^*), $ there exists $({y_1}_*, {y_2}_*)\in\mathcal{N}_{\lambda}^{ - }$ such that
\begin{equation*}
  J({y_1}_*, {y_2}_*) = \inf_{(y_1,y_2)\in\mathcal{N}_{\lambda}^{ - }}J(y_1, y_2).
\end{equation*}
\end{lemma}
\begin{proof}
Choose $(y_1, y_2)\in W_{0}^{1, \mathcal{H}}(\Omega)\times W_{0}^{1, \mathcal{H}}(\Omega)$ satisfied $\int_{\Omega}\vert y_1\vert^{\kappa_1 + 1}\vert y_2\vert^{\kappa_2 + 1}dz > 0$. So, by Lemma \ref{t1t2}, one has $\mathcal{N}_{\lambda}^{ - }\neq\emptyset$ and, by Lemma \ref{bddblow}, $J\vert_{\mathcal{N}_{\lambda}^{ - }}$ is bounded from below. Now let us consider a minimizing sequence $({y_1}_n, {y_2}_n)\in \mathcal{N}_{\lambda}^{ - }$ such that when $n\rightarrow \infty$,
\begin{equation*}
  J({y_1}_n, {y_2}_n)\searrow \inf_{(y_1,y_2)\in\mathcal{N}_{\lambda}^{ - }}J(y_1, y_2).
\end{equation*}
By Lemma \ref{bddblow} both sequences $\{ {y_1}_n\}$ and $\{ {y_2}_n\}$ are bounded in $W_{0}^{1, \mathcal{H}}(\Omega)$ and so, by the reflexivity of $W_{0}^{1, \mathcal{H}}(\Omega)$, there exist $({y_1}_*,{y_2}_*)\in W_{0}^{1, \mathcal{H}}(\Omega)$ and a suitable subsequence, still denoted by $\{( {y_1}_n,{y_2}_n)\}$ such that
\begin{equation*}
  ({y_1}_n, {y_2}_n)\rightharpoonup ({y_1}_*, {y_2}_*)\  {\rm in }\ W_{0}^{1, \mathcal{H}}(\Omega)\times W_{0}^{1, \mathcal{H}}(\Omega),
\end{equation*}
and from Lemma \ref{prop2.2(Aro)}$(iii)$,
\begin{equation*}
  {y_1}_n\rightarrow {y_1}_* \  {\rm in }\  L^{r}(\Omega), \ \ {y_2}_n\rightarrow {y_2}_*\  {\rm in }\  L^{r}(\Omega), \ \ 1\leq r < p^*.
\end{equation*}
Since $1 < p < p^*$ and the assumption $(H_{3})$ holds, \cite{GuoW} gives the fact
\begin{equation*}
  \int_{\Omega}\vert {y_1}_n\vert^{\kappa_1 + 1}\vert {y_2}_n\vert^{\kappa_2 + 1}dz\rightarrow\int_{\Omega}\vert {y_1}_*\vert^{\kappa_1 + 1}\vert {y_2}_*\vert^{\kappa_2 + 1}dz,\ {\rm\ as}\ n\rightarrow \infty.
\end{equation*}
 %then $\|{y_1}_*\|_p^p = 0$, and there exist $n_{0}$such that $\|{y_1}_n\|_p^p < 1$, for $n > n_{0}$. Combining with $({y_1}_n, {y_2}_n)\in \mathcal{N}_{\lambda}^{ - }$, and the sobolev embedding $W_{0}^{1, p}\hookrightarrow L^{r}$, $p\leq r\leq p^*$, and H\"{o}lder inequality, one has
%\begin{equation*}
%\begin{split}
%   (p - 1)C_{22}\|{y_1}_n\|_{r}^p\|{y_2}_n\|_{r}^p\leq & (p - 1)C_{22}( \|{y_1}_n\|_{r}^p + \|{y_2}_n\|_{r}^p) \\
%\leq& (p - 1) \| ({y_1}_n, {y_2}_n) \| \leq \lambda C_{27} \| {y_1}_n\|_{r}^{\kappa_1 + 1}\| {y_2}_n\|_{r}^{\kappa_2 + 1},
%\end{split}
%\end{equation*}
%which means
%\begin{equation*}
% \frac{ (p - 1)C_{22}}{\lambda C_{27}}\leq \| {y_1}_n\|_{r}^{\kappa_1 + 1 - p}\| {y_2}_n\|_{r}^{\kappa_2 + 1 - p}.
%\end{equation*}
%Let $n\rightarrow\infty$, we have $\frac{ (p - 1)C_{22}}{\lambda C_{27}}\leq 0$, its a contradiction. Similarly there is ${y_2}_*\neq 0$.

We claim that ${y_1}_*\neq0$ and ${y_2}_*\neq0$. In fact, since $({y_1}_n, {y_2}_n)\in \mathcal{N}_{\lambda}^{ - }$, we have \eqref{uv < ualpha + 1vbeta + 1}. Let $m_1$ and $m_2$ be defined as the proof of Lemma \ref{3.2},
\begin{equation*}
  m_1 = p^* - \epsilon_{0}(\kappa_1 + 1), \ \ m_2 = \frac{[p^* - \epsilon_{0}(\kappa_1 + 1)](\kappa_2 + 1)}{p^* - (\epsilon_{0} + 1)(\kappa_1 + 1)}.
\end{equation*}
By using Poincar\'{e}'s inequality and \eqref{uvholder}, \eqref{uv < ualpha + 1vbeta + 1} comes to
\begin{equation}\label{u*v*alpha + 1beta + 1neq0}
  \begin{split}
   C_{18}\left( \|{y_1}_n\|_{m_1}^p + \|{y_2}_n\|_{m_2}^p\right) &\leq (p - 1) \| ({y_1}_n, {y_2}_n) \|_{1,p}\\
    &\leq \lambda(\kappa_1 + \kappa_2 + 2)(\kappa_1 + \kappa_2 + 1)\int_{\Omega}\vert {y_1}_n\vert^{\kappa_1 + 1}\vert {y_2}_n\vert^{\kappa_2 + 1}dz\\
  &\leq \lambda (\kappa_1 + \kappa_2 + 2)(\kappa_1 + \kappa_2 + 1)C_{8} \| {y_1}_n\|_{m_1}^{\kappa_1 + 1}\| {y_2}_n\|_{m_2}^{\kappa_2 + 1},
\end{split}
\end{equation}
here, $C_{18}$ is a positive constant. Again since $({y_1}_n, {y_2}_n)\in \mathcal{N}_{\lambda}^{ - }$, we know $({y_1}_n, {y_2}_n)\neq\{(0, 0)\}$. Thus, for some $C_{19}:=(\kappa_1 + \kappa_2 + 2)(\kappa_1 + \kappa_2 + 1)C_{8} >0$, \eqref{u*v*alpha + 1beta + 1neq0} comes to
\begin{equation}\label{chengji}
\begin{split}
  C_{18}\leq \lambda C_{19}\frac{\| {y_1}_n\|_{m_1}^{\kappa_1 + 1}\| {y_2}_n\|_{m_2}^{\kappa_2 + 1}}{ \|{y_1}_n\|_{m_1}^p + \|{y_2}_n\|_{m_2}^p}.
\end{split}
\end{equation}
%Again since $({y_1}_n, {y_2}_n)\in \mathcal{N}_{\lambda}^{ - }$, we know $({y_1}_n, {y_2}_n)\neq\{(0, 0)\}$. If $\{{y_1}_n\}\neq0$, one has
%\begin{equation*}
%   \|{y_1}_n\|_{m_1}^p + \|{y_2}_n\|_{m_2}^p\geq  \|{y_1}_n\|_{m_1}^p>0,
%\end{equation*}
%hence
%\begin{equation*}
%   C_{19}\leq \lambda C_{20}\frac{\| {y_1}_n\|_{m_1}^{\kappa_1 + 1}\| {y_2}_n\|_{m_2}^{\kappa_2 + 1 - p}}{ \|{y_1}_n\|_{m_1}^p + \|{y_2}_n\|_{m_2}^p}\leq \lambda C_{20}\frac{\| {y_1}_n\|_{m_1}^{\kappa_1 + 1}\| {y_2}_n\|_{m_2}^{\kappa_2 + 1}} {\|{y_1}_n\|_{m_1}^p} = \lambda C_{20}\| {y_1}_n\|_{m_1}^{\kappa_1 + 1 - p}\| {y_2}_n\|_{m_2}^{\kappa_2 + 1}.
%\end{equation*}
%Let $n\rightarrow\infty$, we have $C_{19}\leq0$, a contradiction. If $\{{y_2}_n\}\neq\{0\}$,
By the use of Young's inequality and the properties of concave function $t\rightarrow t^{1/p}$, one has for some positive constant $C_{20}$,
\begin{equation}\label{122}
   \begin{split}
      \|{y_1}_n\|_{m_1}^{\frac{\kappa_1 + 1}{\kappa_1 + \kappa_2 + 2}}\|{y_2}_n\|_{m_2}^{\frac{\kappa_2 + 1}{\kappa_1 + \kappa_2 + 2}}&\leq\frac{(\kappa_1 + 1) \|{y_1}_n\|_{m_1}}{\kappa_1 + \kappa_2 + 2} + \frac{(\kappa_2 + 1)\|{y_2}_n\|_{m_2}}{\kappa_1 + \kappa_2 + 2}\\
      & \leq  \|{y_1}_n\|_{m_1} + \|{y_2}_n\|_{m_2}\leq C_{20}(\|{y_1}_n\|_{m_1}^{p} + \|{y_2}_n\|_{m_2}^{p})^{\frac{1}{p}}.
   \end{split}
\end{equation}
Then from \eqref{chengji} and \eqref{122}, we have
\begin{equation}\label{neq11}
\begin{split}
\left( \frac{C_{18}}{\lambda  C_{19} C_{20}^{\kappa_1 + \kappa_2 + 2}}\right)^{\frac{p}{\kappa_1 + \kappa_2 + 2}} \leq  \left( \|{y_1}_n\|_{m_1}^{p} + \|{y_2}_n\|_{m_2}^{p}\right).
\end{split}
\end{equation}
By taking $\lim_{n\rightarrow\infty}$, we know ${y_1}_*\neq 0$ or ${y_2}_*\neq0$. For case of ${y_1}_* = 0$ and ${y_2}_*\neq0$, thus there exists $N\in\mathbb{N}$ large enough such that $\|{y_2}_n\|_{m_2}\neq 0$ for all $n\geq N$. Then
\begin{equation*}
   \|{y_1}_n\|_{m_1}^p + \|{y_2}_n\|_{m_2}^p\geq  \|{y_2}_n\|_{m_2}^p>0 \ \ {\rm for}\  n\geq N,
\end{equation*}
hence for $n\geq N$, \eqref{chengji} comes to
\begin{equation*}
   C_{18}\leq\lambda C_{19}\| {y_1}_n\|_{m_1}^{\kappa_1 + 1}\| {y_2}_n\|_{m_2}^{\kappa_2 + 1 - p}.
\end{equation*}
By taking $\lim_{n\rightarrow\infty}$, we get $C_{18}\leq0$, a contradiction. Thus ${y_1}_*\neq 0$. The proof of the fact that ${y_2}_*\neq 0$ is similar, and we omit it.

Now we take $\lim_{n\rightarrow\infty}$ in \eqref{u*v*alpha + 1beta + 1neq0} to obtain that $\int_{\Omega}\vert {y_1}_*\vert^{\kappa_1 + 1}\vert {y_2}_*\vert^{\kappa_2 + 1}dx > 0$. Then, by Lemma \ref{t1t2}, there exists $t_2 > 0$ such that $t_2({y_1}_*, {y_2}_*)\in \mathcal{N}_{\lambda}^{ - }$.

Now, we prove that, after pass to a subsequence if necessary, $\{( {y_1}_n,{y_2}_n)\}$
converges strongly in $W_{0}^{1, \mathcal{H}}(\Omega)\times W_{0}^{1, \mathcal{H}}(\Omega)$ to $({y_1}_*,{y_2}_*)$. To do it, it is enough to show that
\begin{equation}\label{jixianxiaoyuzhi}
  \liminf\limits_{n \rightarrow \infty}\rho_{\mathcal{H}}(\nabla {y_1}_n) \leq \rho_{\mathcal{H}}(\nabla {y_1}_*)\ {\rm and}\  \ \liminf\limits_{n \rightarrow \infty}\rho_{\mathcal{H}}(\nabla {y_2}_n)  \leq \rho_{\mathcal{H}}(\nabla {y_2}_*).
\end{equation}
Indeed, if \eqref{jixianxiaoyuzhi} holds, we can assume, after pass to a subsequence if necessary, (still denoted by $({y_1}_n,{y_2}_n)$), that
\begin{equation*}
  \lim_{n \rightarrow \infty}\rho_{\mathcal{H}}(\nabla {y_1}_n) \leq \rho_{\mathcal{H}}(\nabla {y_1}_*), \ \lim_{n \rightarrow \infty}\rho_{\mathcal{H}}(\nabla {y_2}_n)  \leq \rho_{\mathcal{H}}(\nabla {y_2}_*).
\end{equation*}
Since $({y_1}_n, {y_2}_n)\rightharpoonup ({y_1}_*, {y_2}_*)\  {\rm in }\ W_{0}^{1, \mathcal{H}}(\Omega)\times W_{0}^{1, \mathcal{H}}(\Omega)$ and since the integrand function of $\rho_{\mathcal{H}}$ is uniformly convex, it follows from the weak lower semi-continuity of the norms and
seminorms and Lemma \ref{prop2.2(liu)}$(iv)$ that (see \cite{Arora}, Page 13)
\begin{equation*}
  \lim_{n \rightarrow \infty}\|({y_1}_n,{y_2}_n)-({y_1}_*,{y_2}_*)\|=\lim_{n \rightarrow \infty}\left[\|\nabla( {y_1}_n-{y_1}_*)\|_{\mathcal{H}} + \|\nabla ( {y_2}_n-{y_2}_*)\|_{\mathcal{H}}\right] =0.
\end{equation*}
Which means $\{( {y_1}_n,{y_2}_n)\}$
converges strongly in $W_{0}^{1, \mathcal{H}}(\Omega)\times W_{0}^{1, \mathcal{H}}(\Omega)$ to $({y_1}_*,{y_2}_*)$.

To prove \eqref{jixianxiaoyuzhi} we proceed by contradiction.
%e.g., \cite{Berger}, Page 31
%Next, we prove $\rho_{\mathcal{H}}(\nabla {y_1}_{n})\rightarrow\rho_{\mathcal{H}}(\nabla {y_1}_{*}) $ and $\rho_{\mathcal{H}}(\nabla y_{n})\rightarrow\rho_{\mathcal{H}}(\nabla y_{*}) $ as $n\rightarrow\infty$.
%
%\begin{itemize}[itemindent = 2.5em]
%  \item [Claim:]\ \ $\liminf\limits_{n \rightarrow \infty}\rho_{\mathcal{H}}(\nabla {y_1}_n) = \rho_{\mathcal{H}}(\nabla {y_1}_*)$ and $\liminf\limits_{n \rightarrow \infty}\rho_{\mathcal{H}}(\nabla {y_2}_n) = \rho_{\mathcal{H}}(\nabla {y_2}_*)$.
%\end{itemize}
Suppose that either $\liminf\limits_{n \rightarrow \infty}\rho_{\mathcal{H}}(\nabla {y_1}_n) > \rho_{\mathcal{H}}(\nabla {y_1}_*)$ or $\liminf\limits_{n \rightarrow \infty}\rho_{\mathcal{H}}(\nabla {y_2}_n) > \rho_{\mathcal{H}}(\nabla {y_2}_*)$. We may have the following three cases:
\begin{itemize}[itemindent = 1.7em]
\item [Case (a):]\ \ $\liminf\limits_{n \rightarrow \infty}\rho_{\mathcal{H}}(\nabla {y_1}_n) > \rho_{\mathcal{H}}(\nabla {y_1}_*), \ \liminf\limits_{n \rightarrow \infty}\rho_{\mathcal{H}}(\nabla {y_2}_n)  = \rho_{\mathcal{H}}(\nabla {y_2}_*)$.
\item [Case (b):]\ \ $\liminf\limits_{n \rightarrow \infty}\rho_{\mathcal{H}}(\nabla {y_1}_n) = \rho_{\mathcal{H}}(\nabla {y_1}_*), \ \liminf\limits_{n \rightarrow \infty}\rho_{\mathcal{H}}(\nabla {y_2}_n)  > \rho_{\mathcal{H}}(\nabla {y_2}_*)$.
\item [Case (c):]\ \ $\liminf\limits_{n \rightarrow \infty}\rho_{\mathcal{H}}(\nabla {y_1}_n) > \rho_{\mathcal{H}}(\nabla {y_1}_*), \ \liminf\limits_{n \rightarrow \infty}\rho_{\mathcal{H}}(\nabla {y_2}_n)  > \rho_{\mathcal{H}}(\nabla {y_2}_*)$.
\end{itemize}

For Case (a), %we have
%\begin{equation*}
%  \begin{split}
%\liminf\limits_{n \rightarrow \infty}& \psi_{({y_1}_n, {y_2}_n)}^{\prime}(t_2) \\
%      = & \liminf\limits_{n \rightarrow \infty} \bigg[t_2^{p - 1}\|({y_1}_n, {y_2}_n)\| + t_2^{q - 1}\int_{\Omega} \eta\vert\nabla {y_1}_n\vert^{q}dx\\
%    & + t_2^{q - 1}\int_{\Omega} \eta\vert\nabla {y_2}_n\vert^{q}dx - t_2^{ - \nu}\int_{\Omega}\left(a_1\vert {y_1}_n \vert^{1 - \nu} + a_2\vert {y_2}_n \vert^{1 - \nu}\right)dx\\
%    & - t_2^{\kappa_1 + \kappa_2 + 1}\lambda(\kappa_1 + \kappa_2 + 2)\int_{\Omega}\vert {y_1}_n\vert^{\kappa_1 + 1}\vert {y_2}_n\vert^{\kappa_2 + 1}dx\bigg]\\
%       >  &t_2^{p - 1} \|({y_1}_*, {y_2}_*)\| + t_2^{q - 1}\int_{\Omega} \eta\vert\nabla {y_1}_*\vert^{q}dx\\
%    &  + t_2^{q - 1}\int_{\Omega} \eta\vert\nabla {y_2}_*\vert^{q}dx - t_2^{ - \nu}\int_{\Omega}\left(a_1\vert {y_1}_* \vert^{1 - \nu} + a_2\vert {y_2}_* \vert^{1 - \nu}\right)dx\\
%    & - t_2^{\kappa_1 + \kappa_2 + 1}\lambda(\kappa_1 + \kappa_2 + 2)\int_{\Omega}\vert {y_1}_*\vert^{\kappa_1 + 1}\vert {y_2}_*\vert^{\kappa_2 + 1}dx\\
%     = &0.\ \ ({\rm because\ of\ }(t_2{y_1}_*, t_2{y_2}_*)\in \mathcal{N}_{\lambda}^{ - })
%  \end{split}
%\end{equation*}
according to $({y_1}_n, {y_2}_n)\in\mathcal{N}_{\lambda}^{ - }$, by Lemma \ref{t1t2}, we have $J({y_1}_n,{y_2}_n)=\max_{t\in(0,\infty)}J(t {y_1}_n, t {y_2}_n)$. By use of the weak lower semi$\mbox{-}$continuity of the norms and seminorms and Lebesgue’s dominated convergence theorem, we have
\begin{equation*}
\inf_{\mathcal{N}_{\lambda}^{ - }}J(y_1, y_2) \leq J(t_2{y_1}_*, t_2{y_2}_*) < \liminf\limits_{n \rightarrow \infty} J(t_2{y_1}_n, t_2{y_2}_n)\leq \liminf\limits_{n \rightarrow \infty} J({y_1}_n, {y_2}_n) = \inf_{\mathcal{N}_{\lambda}^{ - }}J(y_1, y_2).
\end{equation*}
This is a contradiction.

The proof for the Cases (b) and (c) are similar to the given for the Case (a) and we omit them. Thus $ \liminf\limits_{n \rightarrow \infty}\rho_{\mathcal{H}}(\nabla {y_1}_n) \leq \rho_{\mathcal{H}}(\nabla {y_1}_*)$, $\liminf\limits_{n \rightarrow \infty}\rho_{\mathcal{H}}(\nabla {y_2}_n)  \leq \rho_{\mathcal{H}}(\nabla {y_2}_*)$, and then, $({y_1}_n, {y_2}_n)\rightarrow ({y_1}_*, {y_2}_*)\  {\rm in }\ W_{0}^{1, \mathcal{H}}(\Omega)\times W_{0}^{1, \mathcal{H}}(\Omega)$. According to the continuity of $J(y_1, y_2)$, one has $J({y_1}_n, {y_2}_n)\rightarrow J({y_1}_*, {y_2}_*)$, thus $J({y_1}_*, {y_2}_*) = \inf_{\mathcal{N}_{\lambda}^{ - }}J(y_1, y_2)$.
%
%From the Claim we know that we can find two subsequences (still denoted by ${y_1}_n$, ${y_2}_n$) such that $\rho_{\mathcal{H}}(\nabla {y_1}_{n})\rightarrow\rho_{\mathcal{H}}(\nabla {y_1}_{*}) $ and $\rho_{\mathcal{H}}(\nabla y_{n})\rightarrow\rho_{\mathcal{H}}(\nabla y_{*}) $ as $n\rightarrow\infty$. It follows from Proposition \ref{prop2.2(liu)}$(iv)$ that $({y_1}_n, {y_2}_n)\rightarrow ({y_1}_*, {y_2}_*)\  {\rm in }\ W_{0}^{1, \mathcal{H}}(\Omega)\times W_{0}^{1, \mathcal{H}}(\Omega)$. According to the continuity of $J(y_1, y_2)$, one has $J({y_1}_n, {y_2}_n)\rightarrow J({y_1}_*, {y_2}_*)$, thus $J({y_1}_*, {y_2}_*) = \inf_{\mathcal{N}_{\lambda}^{ - }}J(y_1, y_2)$.
\par Since $ \psi_{({y_1}_n, {y_2}_n)}^{\prime\prime}(1) < 0$, by taking the limit as $n\rightarrow \infty$ we obtain  $\psi_{({y_1}_*, {y_2}_*)}^{\prime\prime}(1)\leq 0$. Again since Lemma \ref{3.2}, we know $\mathcal{N}_{\lambda}^{0} = \emptyset$ for $\lambda\in (0, \lambda^*).$ So, $({y_1}_*, {y_2}_*)\in \mathcal{N}_{\lambda}^{ - }$.

\vskip0.5cm
Since $J(\vert y_1\vert, \vert y_2\vert) = J(y_1, y_2)$, we may assume that ${y_1}_*, {y_2}_*$ are nonnegative.
\end{proof}
\begin{lemma}\label{n +  + }
Suppose the assumptions $(H_1)$, $(H_2)$, $(H_{3})$ hold, then for $\lambda\in (0, \lambda^*), $ there exists $({y_1}^*, {y_2}^*)\in\mathcal{N}_{\lambda}^{ + }$ such that
\begin{equation*}
  J({y_1}^*, {y_2}^*) = \inf_{(y_1,y_2)\in\mathcal{N}_{\lambda}^{ + }}J(y_1, y_2),
\end{equation*}
and ${y_1}^*(z), {y_2}^*(z)\geq0$ for a.e. $z\in\Omega$.
\end{lemma}
\begin{proof}
As in the proof of Lemma \ref{N--} there exist $({y_1}^*,{y_2}^*)\in W_{0}^{1, \mathcal{H}}(\Omega)\times W_{0}^{1, \mathcal{H}}(\Omega)$, and a subsequence, still denoted by $({y_1}_n,{y_2}_n)$, such
that $({y_1}_n, {y_2}_n)\rightharpoonup ({y_1}^*, {y_2}^*)\  {\rm in }\ W_{0}^{1, \mathcal{H}}(\Omega)\times W_{0}^{1, \mathcal{H}}(\Omega)$. So, from the weak lower semi$\mbox{-}$continuity of the involved norms and seminorms, and using the Lebesgueís
dominated convergence theorem, as well as Lemma \ref{t3}, we get
%From Lemma \ref{t3}, one has $\mathcal{N}_{\lambda}^{ + }\neq\emptyset$. Again since Lemma \ref{bddblow}, we know the existence of $\inf_{\mathcal{N}_{\lambda}^{ + }}J(y_1, y_2)$.
%
%Let $({y_1}_n, {y_2}_n)\in\mathcal{N}_{\lambda}^{ + }$ such that when $n\rightarrow \infty$,
%\begin{equation*}
% J({y_1}_n, {y_2}_n)\searrow \inf_{(y_1,y_2)\in\mathcal{N}_{\lambda}^{ + }}J(y_1, y_2).
%\end{equation*}
%Here we only give the steps to prove strong convergence $({y_1}_n, {y_2}_n)\rightarrow ({y_1}^*, {y_2}^*) \in W_{0}^{1, \mathcal{H}}(\Omega)\times W_{0}^{1, \mathcal{H}}(\Omega)$. So,  from Lemma \ref{jn +  +  < 0}, weak lower semi$\mbox{-}$continuity of the norms and seminorms and Lebesgue’s dominated convergence theorem, there is
\begin{equation*}
  J({y_1}^*, {y_2}^*)\leq\liminf_{n\rightarrow\infty}J({y_1}_n, {y_2}_n) < 0 = J(0, 0),
\end{equation*}
thus, $({y_1}^*, {y_2}^*)\neq \{(0, 0)\}$, and then, according to Lemma \ref{t3}, we can find $t_2({y_1}^*, {y_2}^*)\in\mathcal{N}_{\lambda}^{ + }$.

Now we prove that, after pass to a further subsequence if necessary, $({y_1}_n, {y_2}_n)\rightharpoonup ({y_1}^*, {y_2}^*)\  {\rm in }\ W_{0}^{1, \mathcal{H}}(\Omega)\times W_{0}^{1, \mathcal{H}}(\Omega)$. Proceeding as in the proof of Lemma \ref{N--}, it is enough to see that each one of the following three cases is impossible,
\begin{itemize}[itemindent = 1.7em]
\item [Case (a):]\ \ $\liminf\limits_{n \rightarrow \infty}\rho_{\mathcal{H}}(\nabla {y_1}_n) > \rho_{\mathcal{H}}(\nabla {y_1}^*), \ \liminf\limits_{n \rightarrow \infty}\rho_{\mathcal{H}}(\nabla {y_2}_n)  = \rho_{\mathcal{H}}(\nabla {y_2}^*)$.
\item [Case (b):]\ \ $\liminf\limits_{n \rightarrow \infty}\rho_{\mathcal{H}}(\nabla {y_1}_n) = \rho_{\mathcal{H}}(\nabla {y_1}^*), \ \liminf\limits_{n \rightarrow \infty}\rho_{\mathcal{H}}(\nabla {y_2}_n)  > \rho_{\mathcal{H}}(\nabla {y_2}^*)$.
\item [Case (c):]\ \ $\liminf\limits_{n \rightarrow \infty}\rho_{\mathcal{H}}(\nabla {y_1}_n) > \rho_{\mathcal{H}}(\nabla {y_1}^*), \ \liminf\limits_{n \rightarrow \infty}\rho_{\mathcal{H}}(\nabla {y_2}_n)  > \rho_{\mathcal{H}}(\nabla {y_2}^*)$.
\end{itemize}

For Case (a), we have
\begin{equation*}
  \begin{split}
\liminf\limits_{n \rightarrow \infty}\  \psi_{({y_1}_n, {y_2}_n)}^{\prime}(t_2) = & \liminf\limits_{n \rightarrow \infty} \bigg\{t_2^{p - 1}\|({y_1}_n, {y_2}_n)\|_{1,p}+t_2^{q - 1}\|(\nabla {y_1}_n, \nabla {y_2}_n)\|_{q,\eta}\\
&- t_2^{ - \nu}\int_{\Omega}\left[\vert {y_1}_n \vert^{1 - \nu} + a_2\vert {y_2}_n \vert^{1 - \nu}\right]dz\\
&- t_2^{\kappa_1 + \kappa_2 + 1}\lambda(\kappa_1 + \kappa_2 + 2)\int_{\Omega}\vert {y_1}_n\vert^{\kappa_1 + 1}\vert {y_2}_n\vert^{\kappa_2 + 1}dz\bigg\}\\
       >  &t_2^{p - 1} \|({y_1}^*, {y_2}^*)\|_{1,p} + t_2^{q - 1}\|(\nabla {y_1}^*,\nabla {y_2}^*)\|_{q,\eta} \\
       &-  t_2^{ - \nu}\int_{\Omega}\left[a_1\vert {y_1}^* \vert^{1 - \nu} + a_2\vert {y_2}^* \vert^{1 - \nu}\right]dz\\
       &- t_2^{\kappa_1 + \kappa_2 + 1}\lambda(\kappa_1 + \kappa_2 + 2)\int_{\Omega}\vert {y_1}^*\vert^{\kappa_1 + 1}\vert {y_2}^*\vert^{\kappa_2 + 1}dz\\
     = &0, \ \ ({\rm because\ of\ }(t_2{y_1}^*, t_2{y_2}^*)\in \mathcal{N}_{\lambda}^{ + }\subset\mathcal{N}_{\lambda} ).
  \end{split}
\end{equation*}
Which means
\begin{equation*}
  \liminf\limits_{n \rightarrow \infty} \psi_{({y_1}_n, {y_2}_n)}^{\prime}(t_2) > \psi_{({y_1}^*, {y_2}^*)}^{\prime}(t_2).
\end{equation*}
Thus there exists $n_{0}\in \mathbb{N}$ such that for all $n > n_{0}$, $\psi_{({y_1}_n, {y_2}_n)}^{\prime}(t_2) > 0$. According to $({y_1}_n, {y_2}_n)\in\mathcal{N}_{\lambda}^{ + }\subset\mathcal{N}_{\lambda}$ and \eqref{psiandphitu}, one has for all $0 < t < 1$,
\begin{equation*}
  \psi_{({y_1}_n, {y_2}_n)}^{\prime}(t) < 0.
\end{equation*}
Thus $t_2 > 1$ and $\psi_{({y_1}^*, {y_2}^*)}^{\prime}(t) < 0$ for all $t\in(0, t_2)$. Again since $(t_2{y_1}^*, t_2{y_2}^*)\in\mathcal{N}_{\lambda}^{ + }$, we have
\begin{equation*}
  \inf_{\mathcal{N}_{\lambda}^{ + }}J(y_1, y_2)\leq J(t_2{y_1}^*, t_2{y_2}^*) \leq J({y_1}^*, {y_2}^*) <  \liminf\limits_{n \rightarrow \infty} J({y_1}_n, {y_2}_n) = \inf_{\mathcal{N}_{\lambda}^{ + }}J(y_1, y_2).
\end{equation*}
This is a contradiction. A similar contradiction is reached also in the Cases (b) and (c). So $({y_1}_n, {y_2}_n)\rightarrow ({y_1}^*, {y_2}^*)$ in $W_{0}^{1, \mathcal{H}}(\Omega)\times W_{0}^{1, \mathcal{H}}(\Omega)$. We argue as in the proof of Lemma \ref{N--} and using $ \psi_{({y_1}_n, {y_2}_n)}^{\prime\prime}(1) > 0$ (because of $({y_1}_n, {y_2}_n)\in \mathcal{N}_{\lambda}^{ + } $), we obtain $({y_1}^*, {y_2}^*)\in\mathcal{N}_{\lambda}^{ + }$ and $J({y_1}^*, {y_2}^*) = \inf_{\mathcal{N}_{\lambda}^{ + }}J(y_1, y_2).$ The proof is complete.
\end{proof}

\vskip0.5cm
Inspired by Lemma 3 of \cite{SWL}, we have the following results.

\begin{lemma}\label{yinhanshun +  + }
 Suppose the assumptions $(H_1)$, $(H_2)$ hold, then for $(y_1, y_2)\in \mathcal{N}_{\lambda}^{ + }$, there exist a small enough positive constant $\epsilon$ and a functional denoted by
\begin{equation*}
  \xi:B_{\epsilon}(0) \rightarrow \mathbb{R}^{ + },
\end{equation*}
which is continuous and satisfied $\xi(0, 0) = 1$, and  for all $(x_1, x_2)\in B_{\epsilon}(0)$,
\begin{equation*}
  \xi(x_1, x_2)(y_1+ x_1, y_2 + x_2)\in\mathcal{N}_{\lambda}^{ + },
\end{equation*}
here, $B_{\epsilon}(0): = \{(y_1 , y_2 \in W_{0}^{1, \mathcal{H}}(\Omega)\times W_{0}^{1, \mathcal{H}}(\Omega)\  \vert\  \|(y_1 , y_2 )\| < \epsilon\}$.
\end{lemma}
\begin{proof}
Given $(y_1, y_2)\in \mathcal{N}_{\lambda}^{ + }$, define the functional $\tilde{H}(x_1, x_2, t):W_{0}^{1, \mathcal{H}}(\Omega)\times W_{0}^{1, \mathcal{H}}(\Omega)\times\mathbb{R}^ + \rightarrow \mathbb{R}$ as
\begin{equation*}
  \tilde{H}(x_1, x_2, t): = t^{\nu}\psi_{(y_1+ x_1, y_2 + x_2)}^{\prime}(t).
\end{equation*}
Because $(y_1+ x_1, y_2 + x_2)\in \mathcal{N}_{\lambda}^{ + }$, we have $\tilde{H}(0, 0, 1) = 0$, $\frac{\partial \tilde{H}}{\partial t}(0, 0, 1) > 0$. Using implicit function theorem to $\tilde{H}$ at $(0, 0, 1)$ (see, Berger \cite{Berger}), there exist $0 < \delta < 1$, $\epsilon > 0$ and a continuous functional $\xi:B_{\epsilon}(0)\rightarrow [1 - \delta
 , 1 + \delta]$ such that $\xi(0, 0) = 1$ and
\begin{equation}\label{xi10 = 1}
\tilde{H}(x_1, x_2, \xi(x_1, x_2)) = 0, \ \ \ \forall\ (x_1, x_2)\in B_{\epsilon}(0).
\end{equation}
 Hence we know $\xi(x_1, x_2)(y_1+ x_1, y_2 + x_2) \in\mathcal{N}_{\lambda}$ for all $\|(x_1 , x_2 )\| < \epsilon$.

Now we prove that $\xi(x_1, x_2)(y_1+ x_1, y_2 + x_2) \in\mathcal{N}_{\lambda}^{ + }$ for any $(x_1,x_2)\in B_{\epsilon}(0)$. Since
\begin{equation*}
\frac{\partial \tilde{H}}{\partial t}(x_1, x_2, t) = \nu t^{\nu - 1}\psi_{(y_1+ x_1, y_2 + x_2)}^{\prime}(t) + t^{\nu}\psi_{(y_1+ x_1, y_2 + x_2)}^{\prime\prime}(t),
\end{equation*}
and $ \xi(x_1, x_2)(y_1+ x_1, y_2 + x_2)\in\mathcal{N}_{\lambda}$, we have, for all $\| (x_1, x_2)\| < \epsilon$,
\begin{equation*}
\frac{\partial \tilde{H}}{\partial t}(x_1, x_2, \xi(x_1, x_2)) = \xi(x_1, x_2)^{\nu}\psi_{(y_1+ x_1, y_2 + x_2)}^{\prime\prime}(\xi(x_1, x_2)).
\end{equation*}
Taking into account that $\frac{\partial \tilde{H}}{\partial t}(0, 0, 1) > 0$, $\xi(0,0)=1$ and that $\xi$ and $\psi_{(y_1+ x_1, y_2 + x_2)}^{\prime\prime}(t)$ are continuous on $B_\epsilon(0)$ and on $[1-\delta,1+\delta]$, $\delta\in(0,1)$, respectively, by
diminishing $\epsilon$ if necessary, we have
\begin{equation*}
 \xi(x_1, x_2)(y_1+ x_1, y_2 + x_2)\in\mathcal{N}_{\lambda}^{ + },
\end{equation*}
for all $(x_1, x_2)\in B_{\epsilon}(0)$.
\end{proof}

\begin{lemma}\label{yinhanshun--}
Suppose the assumptions $(H_1)$, $(H_2)$, $(H_3)$ hold, then for $(y_1,y_2)\in \mathcal{N}_{\lambda}^{ - }$, there exist a small enough positive constant $\epsilon$ and a functional denoted by
\begin{equation*}
  \xi_1:B_{\epsilon}(0)\rightarrow \mathbb{R}^{ + },
\end{equation*}
which is continuous and satisfied $\xi_1(0, 0) = 1$ and for all $(x_1, x_2)\in B_{\epsilon}(0)$,
\begin{equation*}
  \xi_1(x_1, x_2)(y_1 + x_1, y_2 + x_2)\in\mathcal{N}_{\lambda}^{ - }.
\end{equation*}
\end{lemma}
\begin{proof}
  Combining with Lemma \ref{N--} and the similarly proof process of Lemma \ref{yinhanshun +  + }, we can end this proof.
\end{proof}

\begin{lemma}\label{juv < jxithuv}
Suppose the assumptions $(H_1)$, $(H_2)$, $(H_3)$ hold, and $\lambda \in (0, \lambda^*)$, then there exists $\delta^*\in \mathbb{R}^{ + }$ such that for all $(h, w)\in W_{0}^{1, \mathcal{H}}(\Omega)\times W_{0}^{1, \mathcal{H}}(\Omega) $ with $t\in [0, \delta^*]$, $(th,tw)\in B_\epsilon(0)$,
\begin{equation}\label{andianjie}
  J(\xi_1(th, tw)({y_1}_*, {y_2}_*))\leq J(\xi_1(th, tw)({y_1}_* + th, {y_2}_* + tw)).
\end{equation}
\end{lemma}
\begin{proof}
Given $(h, w)\in W_{0}^{1, \mathcal{H}}(\Omega)\times W_{0}^{1, \mathcal{H}}(\Omega) $, define a function $f_{(h, w)}(t):\mathbb{R}^{ + }\rightarrow \mathbb{R}$ by
\begin{equation*}
 \begin{split}
   f_{(h, w)}(t): = &(p - 1) \|({y_1}_* + th, {y_2}_* + tw)\|_{1,p} + (q - 1)\|(\nabla {y_1}_* + t\nabla h,\nabla {y_2}_* + t\nabla w)\|_{q,\eta}\\
   & + \nu\int_{\Omega}\left[a_1\vert {y_1}_* + th \vert^{1 - \nu} + a_2\vert {y_2}_* + tw \vert^{1 - \nu}\right]dz\\
   & - \lambda(\kappa_1 + \kappa_2 + 2)(\kappa_1 + \kappa_2 + 1) \int_{\Omega}\vert {y_1}_* + th\vert^{\kappa_1 + 1}\vert {y_2}_* + tw\vert^{\kappa_2 + 1}dz,
 \end{split}
\end{equation*}
where $({y_1}_*, {y_2}_*)$ given by Lemma \ref{N--}.

Since $({y_1}_*, {y_2}_*)\in \mathcal{N}_{\lambda}^{ - }$, one has
\begin{equation*}
  f_{(h, w)}(0) = \psi_{({y_1}_*, {y_2}_*)}^{\prime\prime}(1) < 0.
\end{equation*}
By the continuity of the function $f_{(h, w)}(t)$, it can be obtained that there exists $ \delta_* > 0$ such that
\begin{equation*}
  \psi_{({y_1}_* + th, {y_2}_* + tw)}^{\prime\prime}(1) = f_{(h, w)}(t) < 0, \ \ \forall\  t\in[0, \delta_*].
\end{equation*}
From Lemma \ref{yinhanshun--}, for $({y_1}_*, {y_2}_*)\in\mathcal{N}_{\lambda}^{ - }$, we can find $ \epsilon>0$, $0 < \delta^* < \delta_*$ and a continuous functional
\begin{equation*}
\xi_1:B_{\epsilon}(0)\rightarrow (0,  + \infty),
\end{equation*}
such that for all $t\in [0, \delta^*]$, $(th, tw)\in B_{\epsilon}(0)$ and $\xi_1(th, tw)({y_1}_* + th, {y_2}_* + tw)\in\mathcal{N}_{\lambda}^{ - }$ with
\begin{equation*}
\lim_{ t\rightarrow 0 ^{ + }}  \xi_1(th, tw)=1.
\end{equation*}
Thus, for $t\in [0, \delta^*]$ with $(th,tw)\in B_{\epsilon}(0)$, one has
\begin{equation*}
  \psi_{({y_1}_* + th, {y_2}_* + tw)}^{\prime\prime}(1) < 0, \ \ \ \psi_{({y_1}_* + th, {y_2}_* + tw)}(\xi_1(th, tw))\geq \psi_{({y_1}_* + th, {y_2}_* + tw)}(1).
\end{equation*}
Hence, let $t\in[0, \delta^*]$, we have
\begin{equation}\label{Ju*v* - J1u*th}
\begin{split}
\psi_{({y_1}_*, {y_2}_*)}&(\xi_1(th, tw))\leq \psi_{({y_1}_*, {y_2}_*)}(1) = J({y_1}_*, {y_2}_*)\\
    & = \inf_{\mathcal{N}_{\lambda}^{ - }}J(y_1, y_2) \leq J(\xi_1(th, tw)({y_1}_* + th, {y_2}_* + tw)).\\
\end{split}
\end{equation}
The proof is complete.
\end{proof}

\begin{remark}\label{reandianjie}
    It is worth mentioning that $J(y_{1*}, y_{2*})$ is not locally minimal see \eqref{andianjie}, but because of $(y_{1*}, y_{2*})\in\mathcal{N}_{\lambda}^{ - }$, $J(y_{1*}, y_{2*})$ is the smallest in cross-section. Therefore, we consider it to have the structure of a saddle point solution.
\end{remark}

\begin{lemma}\label{juv < juxiswv}
Suppose the assumptions $(H_1)$, $(H_2)$ hold, and $\lambda \in (0, \lambda^*)$, then there exists $\delta^*\in \mathbb{R}^{ + }$ such that for all $(h, w)\in W_{0}^{1, \mathcal{H}}(\Omega)\times W_{0}^{1, \mathcal{H}}(\Omega) $ with $t\in [0, \delta^*]$, $(th,tw)\in B_{\epsilon}(0)$,
\begin{equation*}
  J({y_1}^*, {y_2}^*)\leq J({y_1}^* + th, {y_2}^* + tw).
\end{equation*}
\end{lemma}
\begin{proof}
Combining with Lemma \ref{n +  + }, \ref{yinhanshun +  + } and the similarly proof process of Lemma \ref{juv < jxithuv} or Proposition 3.5 in \cite{CPW}, one has for $t\in [0, \delta^*]$,
\begin{equation*}
  \psi_{({y_1}^* + th, {y_2}^* + tw)}^{\prime\prime}(1) > 0, \ \ \ \psi_{({y_1}^* + th, {y_2}^* + tw)}(\xi_1(th, tw))\leq \psi_{({y_1}^* + th, {y_2}^* + tw)}(1).
\end{equation*}
Hence, let $t\in[0, \delta^*]$, we have
\begin{equation*}
\begin{split}
 \psi_{({y_1}^*, {y_2}^*)}(1)& = J({y_1}^*, {y_2}^*)= \inf_{\mathcal{N}_{\lambda}^{ + }}J(y_1, y_2)\leq J(\xi_1(th, tw)({y_1}^* + th, {y_2}^* + tw))\\
  &=\psi_{({y_1}^* + th, {y_2}^* + tw)}(\xi_1(th, tw))\leq \psi_{({y_1}^* + th, {y_2}^* + tw)}(1)=J({y_1}^* + th, {y_2}^* + tw).\\
\end{split}
\end{equation*}
The proof is complete.
\end{proof}

\begin{theorem}\label{n--solution}
Suppose the assumptions $(H_1)$, $(H_2)$, $(H_{3})$ hold, $\lambda\in (0, \lambda^*), $ then $({y_1}_*, {y_2}_*)$ is a positive weak solution of problem $\eqref{eq}$ such that $J({y_1}_*, {y_2}_*)\geq 0$.
\end{theorem}
\begin{proof}

Firstly, we prove that ${y_1}_*, {y_2}_* > 0, ~a.e. ~~z\in \Omega$.

By Lemma 3.8, we have ${y_1}_*, {y_2}_*\geq 0, ~ a.e. ~~z\in \Omega$. Suppose that there exists a set $\mathcal{H}_1\subset \Omega$ such that ${y_1}_* = 0$ for $z\in\mathcal{H}_1$ and $ meas\ \mathcal{H}_1 > 0$ or set $\mathcal{H}_2\subset \Omega$ such that ${y_2}_* = 0$ for $z\in\mathcal{H}_2$ and $ meas\ \mathcal{H}_2 > 0$, ($meas$ stands for the measure). Let $(h, w)\in W_{0}^{1, \mathcal{H}}(\Omega)\times W_{0}^{1, \mathcal{H}}(\Omega) $ with $h\geq0, w\geq0$, and $0 < t < \delta^*$. By the definition of the functional $J(y_1, y_2)$, we have
\begin{equation*}
 \begin{split}
  \frac{1}{t}\big[&J\left(\xi_1(th, tw)({y_1}_* + th, {y_2}_* + tw)\right) - J\left(\xi_1(th, tw)({y_1}_*, {y_2}_*)\right)\big]\\
 = & \frac{\xi_1(th, tw)^p}{pt} \left[ \|({y_1}_* + th, {y_2}_* + tw)\|_{1,p} - \| ({y_1}_*, {y_2}_*)\|_{1,p}\right]\\
& + \frac{\xi_1(th, tw)^{q}}{qt}\int_{\Omega}\eta\left[\vert \nabla ({y_1}_* + th)\vert^{q}dz  - \vert \nabla {y_1}_*\vert^{q}\right]dz \\
& + \frac{\xi_1(th, tw)^{q}}{qt}\int_{\Omega}\eta\left[\vert \nabla ({y_2}_* + th)\vert^{q}dz  - \vert \nabla {y_2}_*\vert^{q}\right]dz\\
& - \frac{\xi_1(th, tw)^{1 - \nu}t^{ - \nu}}{1 - \nu}\int_{\mathcal{H}_1}a_1h^{1 - \nu}dz \\
& - \frac{\xi_1(th, tw)^{1 - \nu}}{(1 - \nu)t}\int_{\Omega\backslash \mathcal{H}_1}a_1\left[({y_1}_* + th)^{1 - \nu} - {y_1}_*^{1 - \nu}\right]dz \\
    & - \frac{\xi_1(th, tw)^{1 - \nu}}{(1 - \nu)t}\int_{\Omega}a_2\left[({y_2}_* + tw)^{1 - \nu} - {y_2}_*^{1 - \nu}\right]dz \\
    & - \frac{\lambda\xi_1(th, tw)^{\kappa_1 + \kappa_2 + 2}}{t}\int_{\Omega}\left[\vert {y_1}_* + th\vert^{\kappa_1 + 1} \vert {y_2}_* + tw\vert^{\kappa_2 + 1} - \vert {y_1}_*\vert^{\kappa_1 + 1}\vert {y_2}_*\vert^{\kappa_2 + 1}\right]dz.
  \end{split}
\end{equation*}
Thus, as $t\rightarrow 0$, by using the L'h\^{o}spital's rule and the fact $0 < \nu < 1$, we have
\begin{equation*}
  \frac{1}{t}\big[J\left(\xi_1(th, tw)({y_1}_* + th, {y_2}_* + tw)\right) - J\left(\xi_1(th, tw)({y_1}_*, {y_2}_*)\right)\big]\rightarrow  - \infty.
\end{equation*}
This is a contradiction to Lemma \ref{juv < jxithuv}. Hence ${y_1}_* > 0$ $a.e.$ $z\in \Omega$. Similarly we have ${y_2}_* > 0 ~a.e.~ z\in\Omega$.

Secondly. we prove that for $(h, w)\in W_{0}^{1, \mathcal{H}}(\Omega)\times W_{0}^{1, \mathcal{H}}(\Omega)$ and $h\geq0$, $w\geq 0$, then
\begin{equation*}
   \left(a_1{y_1}_*^{ - \nu}h, a_2{y_2}_*^{ - \nu}w\right)\in L^{1}(\Omega)\times L^{1}(\Omega), \ \  %\forall\  (h, w)\in W_{0}^{1, \mathcal{H}}(\Omega)\times W_{0}^{1, \mathcal{H}}(\Omega);
\end{equation*}

\begin{equation}\label{zuizhongu}
         \begin{split}
             \int_{\Omega}\vert \nabla {y_1}_*\vert^{p - 2}\nabla {y_1}_*\cdot \nabla hdz &+ \int_{\Omega}\eta\vert \nabla {y_1}_*\vert^{q - 2}\nabla {y_1}_*\cdot \nabla hdz\\
\geq&\int_{\Omega}a_1{y_1}_*^{ - \nu}hdz + \lambda(\kappa_1 + 1)\int_{\Omega}\vert {y_1}_*\vert^{\kappa_1}\vert {y_2}_*\vert^{\kappa_2 + 1} h dz,
        \end{split}
\end{equation}
\begin{equation}\label{zuizhongv}
     \begin{split}
          \int_{\Omega}\vert \nabla {y_2}_*\vert^{p - 2}\nabla {y_2}_*\cdot \nabla wdz &+ \int_{\Omega}\eta\vert \nabla {y_2}_*\vert^{q - 2}\nabla {y_2}_*\cdot \nabla wdz\\
\geq& \int_{\Omega}a_2{y_2}_*^{ - \nu}wdz + \lambda(\kappa_2 + 1)\int_{\Omega}\vert {y_1}_*\vert^{\kappa_1 + 1}\vert {y_2}_*\vert^{\kappa_2} w dz.
     \end{split}
\end{equation}

Given $0\leq h,w\in W_{0}^{1, \mathcal{H}}(\Omega)$, choosing $\{t_n\}\in(0, 1]$ as a decreasing sequence such that $\lim_{n\rightarrow \infty}t_n = 0$. We have that, for $n\in \mathbb{N}$, the function
\begin{equation*}
      u_n(z) = a_1\frac{[{y_1}_*(z) + t_nh(z)]^{1 - \nu} - {y_1}_*(z)^{1 - \nu}}{t_n}
\end{equation*}
is measurable and nonnegative, and for $a.e.\ z\in\Omega$,
\begin{equation*}
   \lim_{n\rightarrow\infty} u_n(z) = (1 - \nu)a_1{y_1}_*(z)^{ - \nu}h(z).
\end{equation*}
Thus
\begin{equation}\label{fatou}
   \int_{\Omega}a_1{y_1}_*(z)^{ - \nu}h(z)dz\leq\frac{1}{1 - \nu}\liminf_{n\rightarrow \infty}\int_{\Omega} u_n(z)dz,
\end{equation}
here, the Fatou's lemma is used.

Applying again Lemma \ref{juv < jxithuv} and letting $w=0$, one has for $n\in \mathbb{N}$ sufficiently large, there is
\begin{equation*}
    \begin{split}
        0\leq & \frac{J(\xi_1(t_nh, 0)({y_1}_* + t_nh, {y_2}_*)) - J(\xi_1(t_nh, 0)({y_1}_*, {y_2}_*))}{t_n}\\
            = & \frac{\xi_1(t_nh, 0)^p}{p}\frac{(\|( {y_1}_* + t_nh)\|_{1, p}^p - \| {y_1}_*\|_{1, p}^p)}{t_n} -  \frac{\xi_1(t_nh, 0)^{\nu}}{1 - \nu}\int_{\Omega}u_ndz \\ & + \frac{\xi_1(t_nh, 0)^{q}}{q}\frac{\int_{\Omega}\eta\vert \nabla ( {y_1}_* + t_nh)\vert^{q}dz -  \int_{\Omega}\eta\vert \nabla {y_1}_*\vert^{q}dz}{t_n} \\
    & - \lambda\xi_1(t_nh, 0)^{\kappa_1 + \kappa_2 + 2}\frac{\int_{\Omega}\vert  {y_1}_* + t_nh\vert^{\kappa_1 + 1} \vert {y_2}_*\vert^{\kappa_2 + 1}dz - \int_{\Omega}\vert  {y_1}_*\vert^{\kappa_1 + 1} \vert {y_2}_*\vert^{\kappa_2 + 1}dz}{t_n}.
    \end{split}
\end{equation*}
Thus
\begin{equation}\label{3.34}
   \begin{split}
         \frac{\xi_1(t_nh, 0)^{\nu}}{1 - \nu}&\int_{\Omega}u_n(z)dz\\
          \leq & \frac{\xi_1(t_nh, 0)^p}{p}\frac{\|( {y_1}_* + t_nh)\|_{1, p}^p - \| {y_1}_*\|_{1, p}^p}{t_n}\\
            &  + \frac{\xi_1(t_nh, 0)^{q}}{q}\frac{\int_{\Omega}\eta\left[\vert \nabla ( {y_1}_* + t_nh)\vert^{q} - \vert \nabla {y_1}_*\vert^{q}\right]dz}{t_n} \\
    & - \lambda\xi_1(t_nh, 0)^{\kappa_1 + \kappa_2 + 2}\frac{\int_{\Omega}\vert  {y_1}_* + t_nh\vert^{\kappa_1 + 1} \vert {y_2}_*\vert^{\kappa_2 + 1} - \vert  {y_1}_*\vert^{\kappa_1 + 1} \vert {y_2}_*\vert^{\kappa_2 + 1}dz}{t_n}.
   \end{split}
\end{equation}
We take $\lim_{n\rightarrow \infty }$ in \eqref{3.34}, and using \eqref{fatou} and the fact that the limit, as $n\rightarrow \infty$, of the right side of \eqref{3.34} exists (and that it is finite), we get that $a_1{y_1}_*^{ - \nu}h\in L^{1}(\Omega)$ for any nonnegative $ h\in W_{0}^{1, \mathcal{H}}(\Omega)$ and that \eqref{zuizhongu} holds. Letting $h = 0$, a similar proof gives that $a_2{y_2}_*^{ - \nu}w\in L^{1}(\Omega)$ for any nonnegative $w\in W_{0}^{1, \mathcal{H}}(\Omega)$ and that \eqref{zuizhongv} hold.

Thirdly, we prove $({y_1}_*, {y_2}_*)$ is a weak solution of problem \eqref{eq}. Given $h_*\in W_{0}^{1, \mathcal{H}}(\Omega)$, $w_*\in W_{0}^{1, \mathcal{H}}(\Omega)$ and replace $h$, $w$ in \eqref{zuizhongu}, \eqref{zuizhongv} with $({y_1}_* + t h_*)_{ + }$, $({y_2}_* + t w_*)_{ + }$, respectively, we have
\begin{equation*}
   \begin{split}
       \int_{\Omega}&\vert \nabla {y_1}_*\vert^{p - 2}\nabla {y_1}_*\cdot \nabla ({y_1}_* + t h_*)_{ + }dz + \int_{\Omega}\eta\vert \nabla {y_1}_*\vert^{q - 2}\nabla {y_1}_*\cdot \nabla ({y_1}_* + t h_*)_{ + }dz\\
& + \int_{\Omega}\vert \nabla {y_2}_*\vert^{p - 2}\nabla {y_2}_*\cdot \nabla ({y_2}_* + t w_*)_{ + }dz + \int_{\Omega}\eta\vert \nabla {y_2}_*\vert^{q - 2}\nabla {y_2}_*\cdot \nabla ({y_2}_* + t w_*)_{ + }dz\\
& - \int_{\Omega}a_1{y_1}_*^{ - \nu}({y_1}_* + t h_*)_{ + }dz - \lambda(\kappa_1 + 1)\int_{\Omega}\vert {y_1}_*\vert^{\kappa_1}\vert {y_2}_*\vert^{\kappa_2 + 1} ({y_1}_* + t h_*)_{ + } dz\\
& - \int_{\Omega}a_2{y_2}_*^{ - \nu}({y_2}_* + t w_*)_{ + }dz - \lambda(\kappa_2 + 1)\int_{\Omega}\vert {y_1}_*\vert^{\kappa_1 + 1}\vert {y_2}_*\vert^{\kappa_2} ({y_2}_* + t w_*)_{ + }dz\geq0.
   \end{split}
\end{equation*}
Thus
\begin{equation}\label{u*hv*wn - }
   \begin{split}
\int_{\Omega}&\vert \nabla {y_1}_*\vert^{p - 2}\nabla {y_1}_* \cdot\nabla ({y_1}_* + t h_*)dz\\
& - \int_{\{{y_1}_* + t h_* < 0\}}\vert \nabla {y_1}_*\vert^{p - 2}\nabla {y_1}_*\cdot \nabla ({y_1}_* + t h_*)dz\\
       & + \int_{\Omega}\eta\vert \nabla {y_1}_*\vert^{q - 2}\nabla {y_1}_*\cdot \nabla ({y_1}_* + t h_*)dz\\
       & - \int_{\{{y_1}_* + t h_* < 0\}}\eta\vert \nabla {y_1}_*\vert^{q - 2}\nabla {y_1}_*\cdot \nabla ({y_1}_* + t h_*)dz\\
       & + \int_{\Omega}\vert \nabla {y_2}_*\vert^{p - 2}\nabla {y_2}_* \cdot\nabla ({y_2}_* + t w_*)dz \\
       &- \int_{\{{y_2}_* + t w_* < 0\}}\vert \nabla {y_2}_*\vert^{p - 2}\nabla {y_2}_*\cdot \nabla ({y_2}_* + t w_*)dz\\
       & + \int_{\Omega}\eta\vert \nabla {y_2}_*\vert^{q - 2}\nabla {y_2}_*\cdot \nabla ({y_2}_* + t w_*)dz\\
       & - \int_{\{{y_2}_* + t w_* < 0\}}\eta\vert \nabla {y_2}_*\vert^{q - 2}\nabla {y_2}_*\cdot \nabla ({y_2}_* + t w_*)dz\\
& - \int_{\Omega}a_1{y_1}_*^{ - \nu}({y_1}_* + t h_*)dz + \int_{\{{y_1}_* + t h_* < 0\}}a_1{y_1}_*^{ - \nu}({y_1}_* + t h_*)dz\\
     & - \lambda(\kappa_1 + 1)\int_{\Omega}\vert {y_1}_*\vert^{\kappa_1}\vert {y_2}_*\vert^{\kappa_2 + 1} ({y_1}_* + t h_*)dz\\
     & +  \lambda(\kappa_1 + 1)\int_{\{{y_1}_* + t h_* < 0\}}\vert {y_1}_*\vert^{\kappa_1}\vert {y_2}_*\vert^{\kappa_2 + 1} ({y_1}_* + t h_*)dz\\
    &  - \int_{\Omega}a_2{y_2}_*^{ - \nu}({y_2}_* + t w_*)dz + \int_{\{{y_2}_* + t w_* < 0\}}a_2{y_2}_*^{ - \nu}({y_2}_* + t w_*)dz\\
     & - \lambda(\kappa_2 + 1)\int_{\Omega}\vert {y_1}_*\vert^{\kappa_1 + 1}\vert {y_2}_*\vert^{\kappa_2} ({y_2}_* + t w_*)dz \\
     &+  \lambda(\kappa_2 + 1)\int_{\{{y_2}_* + t w_* < 0\}}\vert {y_1}_*\vert^{\kappa_1 + 1}\vert {y_2}_*\vert^{\kappa_2} ({y_2}_* + t w_*)dz\geq0.
   \end{split}
\end{equation}
Since $({y_1}_*, {y_2}_*)\in \mathcal{N}_{\lambda}$ and ${y_1}_*(z)>0$, $ {y_2}_*(z) > 0$ for $a.e.\ z\in\Omega$, one has
\begin{equation*}
\begin{split}
0   = & \int_{\Omega}\vert \nabla {y_1}_*\vert^pdz + \int_{\Omega}\eta\vert \nabla {y_1}_*\vert^{q}dz\\
 & - \int_{\Omega}a_1{y_1}_*^{1 - \nu}dz - \lambda(\kappa_1 + 1)\int_{\Omega}\vert {y_1}_*\vert^{\kappa_1 + 1}\vert {y_2}_*\vert^{\kappa_2 + 1} dz\\
 & +  \int_{\Omega}\vert \nabla {y_2}_*\vert^pdz + \int_{\Omega}\eta\vert \nabla {y_2}_*\vert^{q}dz\\
 & - \int_{\Omega}a_2{y_2}_*^{1 - \nu}dz - \lambda(\kappa_2 + 1)\int_{\Omega}\vert {y_1}_*\vert^{\kappa_1 + 1}\vert {y_2}_*\vert^{\kappa_2 + 1} dz.
\end{split}
\end{equation*}
Thus, \eqref{u*hv*wn - } becomes
\begin{equation}\label{u*hv*wn - 2}
\begin{split}
 t \int_{\Omega}&\vert \nabla {y_1}_*\vert^{p - 2}\nabla {y_1}_* \cdot\nabla h_*dz   -  t\int_{\{{y_1}_* + t h_* < 0\}}\vert \nabla {y_1}_*\vert^{p - 2}\nabla {y_1}_*\cdot \nabla h_*dz\\
 + &t \int_{\Omega}\eta\vert \nabla {y_1}_*\vert^{q - 2}\nabla {y_1}_*\cdot \nabla h_*dz  -  t\int_{\{{y_1}_* + t h_* < 0\}}\eta\vert \nabla {y_1}_*\vert^{q - 2}\nabla {y_1}_*\cdot \nabla h_*dz\\
 + & t \int_{\Omega}\vert \nabla {y_2}_*\vert^{p - 2}\nabla {y_2}_* \cdot \nabla w_*dz  -  t \int_{\{{y_2}_* + t w_* < 0\}}\vert \nabla {y_2}_*\vert^{p - 2}\nabla {y_2}_*\cdot \nabla w_*dz\\
 + &t \int_{\Omega}\eta\vert \nabla {y_2}_*\vert^{q - 2}\nabla {y_2}_*\cdot \nabla w_*dz  -  t\int_{\{{y_2}_* + t w_* < 0\}}\eta\vert \nabla {y_2}_*\vert^{q - 2}\nabla {y_2}_*\cdot \nabla {y_2}_*dz\\
 - &t\int_{\Omega}a_1{y_1}_*^{ - \nu}h_*dz  -  t \lambda(\kappa_1 + 1)\int_{\Omega}\vert {y_1}_*\vert^{\kappa_1}\vert {y_2}_*\vert^{\kappa_2 + 1}h_*dz\\
 - &t\int_{\Omega}a_2{y_2}_*^{ - \nu}w_*dz  -  t \lambda(\kappa_2 + 1)\int_{\Omega}\vert {y_1}_*\vert^{\kappa_1 + 1}\vert {y_2}_*\vert^{\kappa_2}w_*dz \geq 0.
\end{split}
\end{equation}
Dividing \eqref{u*hv*wn - 2} by $t$, passing to the limit $t\rightarrow 0$, we have $meas\{{y_1}_* + t h_* < 0\}\rightarrow 0$ and $meas\{{y_2}_* + t w_* < 0\}\rightarrow 0$. Thus
\begin{equation*}
  \begin{split}
              \int_{\Omega}&\vert \nabla {y_1}_*\vert^{p - 2}\nabla {y_1}_*\cdot \nabla h_*dz + \int_{\Omega}\eta\vert \nabla {y_1}_*\vert^{q - 2}\nabla {y_1}_*\cdot \nabla h_*dz\\
              & +  \int_{\Omega}\vert \nabla {y_2}_*\vert^{p - 2}\nabla {y_2}_*\cdot \nabla w_*dz + \int_{\Omega}\eta\vert \nabla {y_2}_*\vert^{q - 2}\nabla {y_2}_*\cdot \nabla w_*dz\\
     & -  \int_{\Omega}a_1{y_1}_*^{ - \nu}h_*dz + \lambda(\kappa_1 + 1)\int_{\Omega}\vert {y_1}_*\vert^{\kappa_1}\vert {y_2}_*\vert^{\kappa_2 + 1} h_* dz\\
     & - \int_{\Omega}a_2{y_2}_*^{ - \nu}w_*dz + \lambda(\kappa_2 + 1)\int_{\Omega}\vert {y_1}_*\vert^{\kappa_1 + 1}\vert {y_2}_*\vert^{\kappa_2} w_* dz \geq 0.
   \end{split}
\end{equation*}
 Since the arbitrary of $ h_*$ and $ w_*$, then the above inequality is equal to 0. Hence $({y_1}_*, {y_2}_*)$ is a positive solution of system \eqref{eq} and from Lemma \ref{lem3.4}, there holds $J({y_1}_*, {y_2}_*)> 0$.
\end{proof}

\begin{theorem}\label{maintheorem}
Suppose the assumptions $(H_1)$, $(H_2)$, $(H_3)$ hold, then there exists a positive constant $\lambda^*$ such that for all $\lambda \in (0, \lambda^*)$, system \eqref{eq} has at least two positive weak solutions $({y_1}^*, {y_2}^*), ({y_1}_*, {y_2}_*)\in W_{0}^{1, \mathcal{H}}(\Omega)\times W_{0}^{1, \mathcal{H}}(\Omega)$ such that $J({y_1}^*, {y_2}^*) < 0< J({y_1}_*, {y_2}_*)$.
\end{theorem}
\begin{proof}
Proceeding as in the proof of Theorem \ref{n--solution}, but using now Lemma \ref{juv < juxiswv} instead of Lemma \ref{juv < jxithuv}, we obtain that $({y_1}^*, {y_2}^*)\in \mathcal{N}_{\lambda}^{ + }$ is a positive weak solution of the problem $\eqref{eq}$ and from Lemma \ref{jn +  +  < 0}, one has $J({y_1}^*, {y_2}^*) < 0$. Also, by Theorem \ref{n--solution}, $({y_1}_*, {y_2}_*)\in \mathcal{N}_{\lambda}^{-}$ is another positive weak solution of system \eqref{eq} and, by Lemma \ref{lem3.4}, $J({y_1}_*, {y_2}_*)> 0$. Thus the proof is complete.
\end{proof}

\vskip1cm

 Next, we generalize the results of the system \eqref{eq} to the case of $m>2$.

 Let $Y =\left[ W_{0}^{1, \mathcal{H}}(\Omega)\right]^{m}$ equipped with norm
\begin{equation*}
  \|(y_1, y_2, \cdots, y_m )\|_{Y,1} =\left( \|y_1\|_{1, p}^p + \|y_2\|_{1, p}^p + \cdots + \|y_m\|_{1, p}^p\right)^{\frac{1}{p}}.
\end{equation*}
\begin{equation*}
\begin{split}
    \|(y_1, y_2, \cdots, y_m )\|_{Y,2}= \|\nabla {y}_1\|_{\mathcal{H}} + \|\nabla {y}_2\|_{\mathcal{H}} + \cdots + \|\nabla {y}_m\|_{\mathcal{H}},
\end{split}
\end{equation*}
and
\begin{equation*}
\begin{split}
  \| (y_1, y_2, \cdots, y_m )\|_{Y,3}&=\left(\| y_1\|_{q,\eta}^{q}+\| y_2\|_{q,\eta}^{q} + \cdots + \| y_m\|_{q,\eta}^{q}\right)^{\frac{1}{q}}\\
  &=\left(\int_{\Omega} \eta\vert {y}_1\vert^{q}dz + \int_{\Omega} \eta\vert {y}_2\vert^{q}dz + \cdots +  \int_{\Omega} \eta\vert {y}_m\vert^{q}dz\right)^{\frac{1}{q}}.
  \end{split}
\end{equation*}
Let the energy functional $J:Y\rightarrow \mathbb{R}$ be defined by

\begin{align*}
   J(y_1, y_2, \cdots, y_m) =&\frac{1}{p}\|(y_1, y_2, \cdots, y_m )\|_{Y,1}^{p}  + \frac{1}{q} \| (\nabla y_1, \nabla y_2, \cdots,\nabla y_m )\|_{Y,3}^{q} \\
   &-  \sum_{i = 1}^{n} \frac{1}{1 - \nu}\int_{\Omega}a_i\vert y_i\vert^{1 - \nu}dz\\
   &- \lambda\int_{\Omega}\vert y_1\vert^{\kappa_1 + 1} \vert y_2\vert^{\kappa_2 + 1}\cdots  \vert y_m\vert^{\kappa_m + 1}dz.
\end{align*}
\begin{theorem}\label{fangchengzujie}
Suppose the assumptions $(H_{1})$, $(H_{2})$, $(H_{3})$ hold, then there exists a positive constant $\lambda^{**}$ such that for all $\lambda \in (0, \lambda^{**})$, system \eqref{eq} has at least two positive weak solutions
\begin{equation*}
 (y_1^*, y_2^*, \cdots, y_m^*), (y_{1*}, y_{2*}, \cdots, y_{m*})\in Y
\end{equation*}
 such that
\begin{equation*}
  J(y_1^*, y_2^*, \cdots, y_m^*) < 0 < J(y_{1*}, y_{2*}, \cdots, y_{m*}).
\end{equation*}
\end{theorem}
\begin{proof}
By checking the proof of Theorem \ref{maintheorem}, we only need to generalize \eqref{uvualpha + 1vbeta + 1} of Lemma \ref{3.2} to the following inequality,
\begin{equation}\label{u1un < alpha1n}
  \left(\sum_{i = 1}^{m}\kappa_i + m\right)\int_{\Omega}\prod_{i = 1}^{m}\vert y_i\vert^{\kappa_i + 1}dz\leq  C_{21}\| (y_1, y_2, \cdots, y_m) \|_{Y,1}^{\sum_{i = 1}^{m}\kappa_i + m}.
\end{equation}
In fact, from assumption $(H_6)$ and (4.1)$\mbox{-}$(4.4) of Lemma 4.1 in \cite{RasoC}, one has
\begin{equation}\label{nu1u2u3un}
    \begin{split}
    \left(\sum_{i = 1}^{n}\kappa_i + m\right)\int_{\Omega}\prod_{i = 1}^{m}\vert y_i\vert^{\kappa_i + 1}dz\leq &C_{22} \prod_{i = 1}^{m} \|y_i\|_{1, p}^{\kappa_i + 1}.
    \end{split}
\end{equation}
For case of $i = 3$. Combining with \eqref{uvualpha + 1vbeta + 1}, there is
\begin{equation}\label{u1u2u3 < alpha123}
 \begin{split}
\left(\|y_1\|_{1, p}^{\kappa_1 + 1}\|y_2\|_{1, p}^{\kappa_2 + 1} \|y_{3}\|_{1, p}^{\kappa_{3} + 1}\right)&^{\frac{1}{\kappa_1 + \kappa_2 + \kappa_{3} + 3}} \\
\leq C_{23}(\|y_1\|_{1, p}^p& + \|y_2\|_{1, p}^p)^{\frac{\kappa_1 + \kappa_2 + 2}{p(\kappa_1 + \kappa_2 + \kappa_{3} + 3)}}\|y_{3}\|_{1, p}^{\frac{\kappa_{3} + 1}{\kappa_1 + \kappa_2 + \kappa_{3} + 3}}.
 \end{split}
\end{equation}
By Young's inequality and the properties of concave function $t\rightarrow t^{1/p}$, one has
\begin{equation*}
\begin{split}
(\|y_1\|_{1, p}^p + &\|y_2\|_{1, p}^p)^{\frac{\kappa_1 + \kappa_2 + 2}{p(\kappa_1 + \kappa_2 + \kappa_{3} + 3)}}\|y_{3}\|_{1, p}^{\frac{\kappa_{3} + 1}{\kappa_1 + \kappa_2 + \kappa_{3} + 3}} \\ \leq & \frac{(\kappa_1 + \kappa_2 + 2)(\|y_1\|_{1, p}^p + \|y_2\|_{1, p}^p)^{\frac{1}{p}}}{\kappa_1 + \kappa_2 + \kappa_{3} + 3} + \frac{(\kappa_{3} + 1)\|y_{3}\|_{1, p}}{\kappa_1 + \kappa_2 + \kappa_{3} + 3}\\
    \leq & (\|y_1\|_{1, p}^p + \|y_2\|_{1, p}^p)^{\frac{1}{p}} + \|y_{3}\|_{1, p}\\
    \leq &C_{24} \left(\|y_1\|_{1, p}^p + \|y_2\|_{1, p}^p + \|y_{3}\|_{1, p}^p\right)^{\frac{1}{p}}.
\end{split}
\end{equation*}
Then, \eqref{u1u2u3 < alpha123} comes to
\begin{equation*}
\begin{split}
   \|y_1\|_{1, p}^{\kappa_1 + 1}\|y_2\|_{1, p}^{\kappa_2 + 1} \|y_{3}\|_{1, p}^{\kappa_{3} + 1}\leq & C_{19}  (\|y_1\|_{1, p}^p + \|y_2\|_{1, p}^p)^{\frac{\kappa_1 + \kappa_2 + 2}{p}}\|y_{3}\|_{1, p}^{\kappa_{3} + 1} \\
    \leq & C_{25} (\|y_1\|_{1, p}^p + \|y_2\|_{1, p}^p + \|y_{3}\|_{1, p}^p)^{\frac{\kappa_1 + \kappa_2 + \kappa_{3} + 3}{p}}.
\end{split}
\end{equation*}
Similarly, using the recursive method, we can see that for case of $i=m$, one has
\begin{equation}\label{nu1u2u3un2}
   \prod_{i = 1}^{m} \|y_i\|_{1, p}^{\kappa_i + 1}\leq C_{26}\|(y_1, y_2, \cdots, y_m)\|_{Y,1}^{\sum_{i = 1}^{m}\kappa_i + m},
\end{equation}
here $C_{22}, C_{23}, C_{24}, C_{25}, C_{26}$ are positive constants. Combining \eqref{nu1u2u3un} with \eqref{nu1u2u3un2}, we can deduce \eqref{u1un < alpha1n} and the proof of Theorem \ref{fangchengzujie} is complete.
\end{proof}

\section{Conclusion}
In this paper, we generalize the double phase problem from a single equation to a system with singular and superlinear terms and by using of the Nehari method, the existence of two positive weak solutions is obtained. It is worth mentioning that the system we are considering is actually a special case of the following system can be considered in the future,
\begin{equation}\label{eq2}
  \begin{cases}
    - \Delta_{p}y - {\rm div}(\eta\lvert \nabla {y}\rvert^{q - 2}\nabla y) = f(z, y), 
\end{cases}
\end{equation}
where $y\in \left[W^{1,\mathcal{H}}(\Omega)\right]^{m}$; $p,q\in \mathbb{R}^{m}$, $p_i<q_i$ ($i=1,\cdots,m$) and $\eta \in \overline{\Omega} \rightarrow \mathbb{R}^{m}$ with $\eta_i(z)\geq 0$ $(i=1,\cdots,m)$ for $a.e.$ $z\in \Omega$. Additionally, double phase systems with more kinds of nonlinear terms are also worthy of further study.
\subsection*{Acknowledgements} The authors would like to sincerely thank the referees for their valuable comments,
which helped improve the initial manuscript.
\subsection*{Funding}
This research was supported by NSFC grant number 11571207, Shandong Provincial Natural Science Foundation number ZR2021MA064 and the Taishan Scholar project.

\section*{Declarations}
\subsection*{Conflict of interest} The authors declare that they have no conflict of interest.
\subsection*{Data Availability Statement}
Data sharing not applicable to this article as no datasets were generated or analyzed during the current study.
\section*{References}
\begingroup
\renewcommand{\section}[2]{}

\end{document}